\documentclass[12pt]{amsart}
\pdfoutput=1
\usepackage[left=1in,top=1in,right=1in,bottom=1in,letterpaper]{geometry}
\usepackage[all]{xy}
\usepackage{epsfig}
\usepackage{amsmath}
\usepackage{amssymb}
\usepackage{amscd}
\usepackage{bm}
\usepackage{bbm}
\def\bbone{{\mathbbm{1}}}
\usepackage{url}
\usepackage{verbatim} 
\usepackage{color}
\usepackage{epsfig}
\usepackage{stmaryrd}
\usepackage{amsthm}
\usepackage{pifont}
\usepackage{mathrsfs}
\usepackage{wasysym}
\usepackage{hyperref}
\usepackage{graphicx}

\title{Truchet tilings and renormalization}
\author[W. Patrick Hooper]{W. Patrick Hooper}
\thanks{Supported by N.S.F. Postdoctoral Fellowship DMS-0803013}
\address{
The City College of New York\\
New York, NY, USA 10031}
\email{whooper@ccny.cuny.edu}

\ifdefined\theorem
\else
\newtheorem{theorem}{Theorem}
\newtheorem{proposition}[theorem]{Proposition}
\newtheorem{lemma}[theorem]{Lemma}
\newtheorem{remark}[theorem]{Remark}
\newtheorem{corollary}[theorem]{Corollary}

\theoremstyle{definition}
\newtheorem{definition}[theorem]{Definition}

\fi
\ifdefined\openquestion
\else
\newtheorem{openquestion}[theorem]{Open Question}
\fi

\ifdefined\question
\else
\newtheorem{question}[theorem]{Question}
\fi

\newlength{\savearraycolsep}
\newenvironment{narrowarray}[2]%
	{\setlength{\savearraycolsep}{\arraycolsep}%
	\setlength{\arraycolsep}{#1}%
	\begin{array}{#2}}%
	{\end{array}\setlength{\arraycolsep}{\savearraycolsep}}

\newcommand{\set}[2]{{\{#1~:~\textrm{#2}\}}}

\def\N{\mathbb{N}}%
\def\R{\mathbb{R}}%
\def\Z{\mathbb{Z}}%
\def\0{{\mathbf{0}}}
\def\1{{\mathbf{1}}}

\def\bd{{\mathbf{d}}}
\def\be{{\mathbf{e}}}

\def\g{{\mathbf{g}}}

\def\m{{\mathbf{m}}}
\def\bn{{\mathbf{n}}}
\def\bp{{\mathbf{p}}}

\def\s{{\mathbf{s}}}

\def\v{{\mathbf{v}}} %
\def\bw{{\mathbf{w}}}
\def\x{{\mathbf{x}}}
\def\y{{\mathbf{y}}}

\def\bpi{{\mathbf{\pi}}}

\def\sA{{\mathcal{A}}}

\def\sC{{\mathcal{C}}}
\def\sD{{\mathcal{D}}}
\def\sE{{\mathcal{E}}}
\def\sG{{\mathcal{G}}}
\def\sI{{\mathcal{I}}}

\def\sL{{\mathcal{L}}}
\def\sM{{\mathcal{M}}}

\def\sO{{\mathcal{O}}}

\def\sR{{\mathcal{R}}}

\def\sT{{\mathcal{T}}}

\def\sW{{\mathcal{W}}}

\def\bbT{{\mathbb{T}}}

\renewcommand{\d}[1]{\ensuremath{\textit{d$#1$}}}%
\makeatletter
\def\imod#1{\allowbreak\mkern10mu({\operator@font mod}\,\,#1)}
\makeatother

\def\and{{\quad \textrm{and} \quad}}

\newcommand{\sm}{\smallsetminus}

\def\omegaalt{\omega^{\textrm{alt}}}
\def\id{{\textrm{id}}}
\def\barK{{\overline{K}}}
\newcommand{\ret}[2]{{R}_{#1}(#2)}
\def\NUC{{\mathit NUC}}
\def\NS{{\mathit NS}}
\def\Step{{\mathcal S}}

\newif\ifdraft\drafttrue
\draftfalse

\newcommand{\name}[1]{\label{#1}{\ifdraft{\textcolor{blue}{\{\textrm{#1}\}}}\else\ignorespaces\fi}}

\newcommand{\tiling}[1]{{[#1]}}

\begin{document}
\maketitle
\begin{abstract}
The Truchet tiles are a pair of square tiles decorated by arcs. When the tiles are pieced together to form a Truchet tiling, these arcs join to form a family of simple curves in the plane. 
We consider a family of probability measures on the space of Truchet tilings.
Renormalization methods are used to investigate the probability that a curve in a Truchet tiling is closed.
\end{abstract}

\section{Preliminary remarks}
This article was written before I was aware of the connection to the corner percolation model studied by G{\'a}bor Pete \cite{Pete08}.
The main result in this paper duplicates results in that paper, albeit by completely different methods. The author believes that the approach in this paper can be used to obtain stronger results than indicated in this paper, such as stronger information of the probability of a curve to have length $L$ for any constant $L$.
I hope to soon release a new version addressing this point of view.

The point of view of this paper also connects to rectangle exchange maps. This
connection is developed in \cite{H12}.

\section{Introduction}
\name{sect:introduction}
The {\em Truchet tiles} are the two $1 \times 1$ squares decorated by arcs as below.
\begin{center}
\includegraphics[height=0.5in]{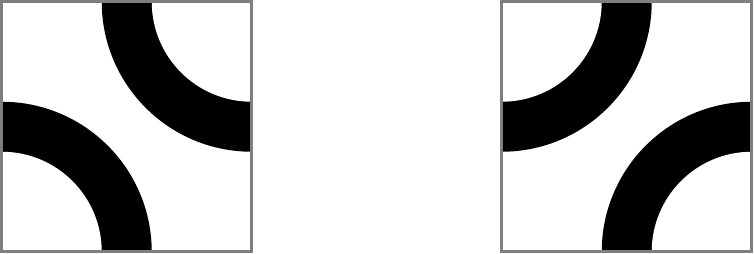}
\end{center}
We call the left tile $T_{-1}$ and the right tile $T_{1}$. The subscripts were chosen to indicate the slope of segments formed
by straightening the arcs to segments.

Given a function $\tau:\Z^2 \to \{\pm 1\}$, the {\em Truchet tiling determined by $\tau$} is the tiling of the plane formed by placing
a copy of the tile $T_{\tau(m,n)}$ centered at the point $(m,n)$ for each $(m,n) \in \Z^2$. 
We denote this tiling by $\tiling{\tau}$. 
Variations of these tilings were first studied for aesthetic reasons by S\'ebastien Truchet in the early 1700s \cite{Truchet}, and this version of tiles were first described by Smith and Boucher
 \cite{Smith87}. 

The arcs on the tiles of a Truchet tiling join to form a disjoint collection of simple curves in the plane. 
See Figure \ref{fig:renormalize}. We call these the {\em curves} of the tiling. Each curve is either closed or bi-infinite. A natural question
to ask is ``how prevalent are the closed curves?'' 
This was asked by Pickover for Truchet tilings which are random in the sense that for each $m,n \in \Z$, 
$\tau(m,n)$ is determined by the flip of a fair coin \cite{Pickover89}.

In this paper, we consider Truchet tilings that arise from functions $\tau_{\omega, \omega'}:\Z^2 \to \{\pm 1\}$ defined in terms of two functions $\Z \to \{\pm 1\}$,
$n \mapsto \omega_n$ and $n \mapsto \omega'_n$. These are the Truchet tilings determined by the function
\begin{equation}
\name{eq:tau omega}
\tau_{\omega, \omega'}(m,n)=\omega_m \omega'_n.
\end{equation}
An example of a portion of such a tiling is shown in Figure \ref{fig:renormalize}.
We will analyze these tilings with techniques coming from the theory dynamical systems. We will show that closed curves curves are highly prevalent in some families of tilings , while not so prevalent in other families.
For the tilings of the form $\tiling{\tau_{\omega, \omega'}}$, we will show that these tilings are renormalizable in the sense of dynamical systems.
Renormalization powers the strongest results in this paper. We explain the motivating connections to dynamical systems in the Motivation Section below. 

To make the notion of prevalence rigorous, observe that $\omega$ and $\omega'$ are naturally elements of the full two-sided shift on the alphabet $\{\pm 1\}$, denoted $\Omega_{\pm}=\Z^{\{\pm 1\}}$. 
The {\em shift map} $\sigma:\Omega_{\pm} \to \Omega_{\pm}$ is defined by $\big(\sigma(\omega)\big)_n=\omega_{n+1}$.   
We may think of choosing our $\omega$ and $\omega'$ at random according to some shift-invariant probability measures $\mu$ and $\mu'$ on $\Omega_{\pm}$.
We then choose an edge $e$ of the tiling by squares centered at the integer points (e.g., take $e$ to join $(\frac{-1}{2},\frac{1}{2})$ to $(\frac{1}{2},\frac{1}{2})$).
We ask ``what is $\mu \times \mu' \big(A)$ where $A$ is the collection of $(\omega, \omega') \in \Omega_{\pm}$ where
the curve through $e$ of the tiling $\tiling{\tau_{\omega,\omega'}}$ is closed?''
A simple argument shows that shift invariance of $\mu$ and $\mu'$ guarantees that this number is independent of the choice of $e$.

In general, we have the following result, which often prevents the tilings with a closed curve through $e$ from being full measure.
\begin{theorem}[Drift Theorem]
\name{thm:drift}
Suppose $\mu$ and $\mu'$ are shift-invariant probability measures on $\Omega_\pm$ satisfying
$$p=\int_{\Omega_\pm} \omega_0~d\mu(\omega) \and q=\int_{\Omega_\pm} \omega'_0~d\mu'(\omega').$$
Then the $\mu \times \mu'$ measure of the set of pairs $(\omega, \omega')$ such that 
the curve of $\tiling{\tau_{\omega,\omega'}}$ through edge $e$ is bi-infinite is at least $\max \{|p|, |q|\}$. 
\end{theorem}
The above result is independent from our renormalization arguments. We prove it and state a stronger version of the result in section \ref{sect:drift}. 

The main case of interest for this paper is when the measures $\mu$ and $\mu'$ are the stationary (shift-invariant) measures associated to some Markov chain with the state space $\{\pm 1\}$. The above theorem indicates that for such measures, the probability that the curve through $e$ is closed is less than one whenever $1$ and $-1$ occur with different probabilities. But, there is a one-parameter family of stationary measures associated to Markov chains with the property that $1$ and $-1$ occur with equal probability. The following theorem addresses these cases.

Fix real constants $p$ and $p'$ with $0<p,p'<1$. 
We describe a method of randomly choosing $\omega$ and $\omega'$. Choose $\omega_0, \omega'_0 \in \{\pm 1\}$ randomly according to the flip of a fair coin. Define the remaining values according to the rule that for all integers $n \geq 0$,
$\omega_{n+1}=\omega_n$ with probability $p$ and $\omega_{-n-1}=\omega_{-n}$ with probability $p$. Do the same for $\omega'$ with probability $p'$.

\begin{theorem}
\name{thm:random tilings}
Let $e$ be an edge of the square tiling of the plane as above. Fix $p$ and $p'$, and define $\omega$ and $\omega'$ randomly as above. With probability $1$, the
curve through $e$ of $\tiling{\tau_{\omega,\omega'}}$ is closed.
\end{theorem}

\begin{figure}
\begin{center}
\includegraphics[width=5.5in]{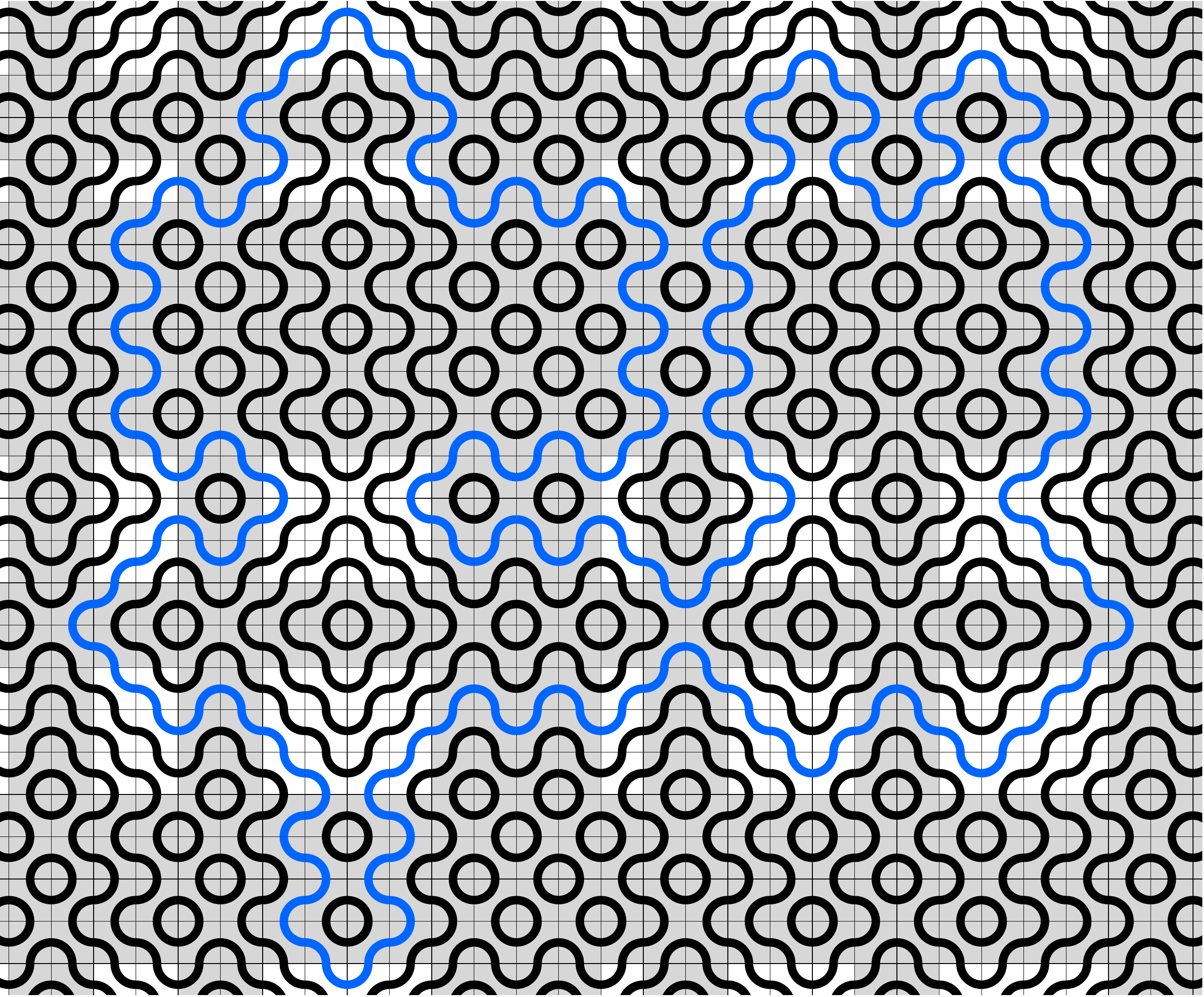}
\caption{A single curve has been highlighted in a tiling determined from $\omega$ and $\omega'$ chosen at random as described above Theorem \ref{thm:random tilings} for some $p$ and $p'$.}
\name{fig:renormalize}
\end{center}
\end{figure}

The above theorem is proven with the renormalization techniques mentioned above. Much of this technique applies in general to pairs of shift-invariant measures $\mu$ and $\mu'$. 

We will informally explain how this notion of renormalization works. Fix $\omega, \omega' \in \Omega_\pm$ and consider the tiling $[\tau_{\omega, \omega'}]$ described by Equation \ref{eq:tau omega}. 
Generally, there is a collection of rows and columns of the tiling following statements hold, letting $Y$ be the union of all rows and columns contained in the collection. (An example of $Y$ consists of the union of gray squares in figure \ref{fig:renormalize}.)
\begin{enumerate}
\item A curve of the tiling is contained entirely in $Y$ if and only if the curve
closes up after visiting only four squares.
\item All other curves of the tiling intersect both $Y$ and its compliment.
\item When a curve enters $Y$ from the left (respectively, the right) through a vertical edge $e$ in $\partial Y$, it exits through the nearest vertical edge in $\partial Y$ directly to the right (resp. the left) of $e$.
\item When a curve enters $Y$ from below (resp. above) through a horizontal edge $e$ in $\partial Y$, it exits through the nearest horizontal edge in $\partial Y$ directly above (resp. below) $e$.
\end{enumerate}
Assuming these statements, we can then form a new tiling by collapsing all the columns contained in $Y$, identifying the pairs of edges mentioned in statement (3). Then we collapse the rows in $Y$, identifying the pairs of edges mentioned in statement (4) in a similar manner. So from $[\tau_{\omega, \omega'}]$ we have obtained a new Truchet tiling, say $[\tau_{\eta, \eta'}]$. Now let $\gamma$ be a curve in the tiling $[\tau_{\omega, \omega'}]$ which intersects both $Y$ and its compliment. Consider the collection of arcs of $\gamma$ which are contained in the compliment of $Y$.
Because of statements (3) and (4), the collapsing process takes these arcs and reassembles them to make a new
connected curve $\gamma'$ in the tiling $[\tau_{\eta, \eta'}]$.
The first two statements imply that we have removed all loops of length $4$, and reduced the length of all closed loops.
So, a loop is closed if and only if it vanishes under some finite number of applications of this collapsing process.

Under certain natural assumptions on $\omega$ and $\omega'$, the resulting tiling is determined by another pair $\eta=c(\omega) \in \Omega_\pm$ and $\eta'=c(\omega') \in \Omega_\pm$. This map $c$ is well defined on some Borel subset $C \subset \Omega_\pm$, and we call 
$c:C \to \Omega_\pm$ the {\em collapsing map}. The collapsing map acts on shift invariant probability measures $\mu$ via
$\mu \mapsto \frac{1}{\mu(C)} \mu \circ c^{-1}$. 

We can immediately observe for instance that $\mu$-a.e. curve is closed (in the sense of the above theorems),
if and only if $\mu \circ c^{-1}$-a.e. curve is closed. Typically much more than this is true.
We define $\omegaalt \in \Omega_\pm$ by $\omegaalt_n=(-1)^n$. 
Assuming that for all integers $n>0$, the measures $\mu \circ c^{-n}$ and $\mu' \circ c^{-n}$ never has an atom at $\omegaalt$, there is a limiting formula for the probability that a curve is closed in terms of the measures 
$\mu \circ c^{-n}$ and $\mu' \circ c^{-n}$ and the behavior of a cocycle acting on a function space.
(See Corollary \ref{cor:approach}.) In the case of the stationary measures associated to a Markov chain
as implicitly discussed in Theorem \ref{thm:random tilings}, the action of this cocycle leaves invariant
a finite dimensional subspace. Understanding the action of this cocycle restricted to this subspace allows us to prove this Theorem \ref{thm:random tilings}.

\subsection{Motivation}

The original motivation for studying Truchet tilings here arose from connections between Truchet tilings and certain
low complexity dynamical systems, such as interval exchange maps and polygon exchange maps. In the paper \cite{H12}, the connection between Truchet tilings and a family of polygon exchange maps will be discussed in depth.
We will briefly explain this connection here because it motivated this work.

A {\em polygon exchange map} is a map $X \to X$, where $X$ is a union of polygons, which is piecewise continuous on polygonal pieces, acts as a translation on each piece, and has an image of full area in $X$. (There may be some ambiguity of definition on the boundaries of the pieces.) 

For a more specific example, let $F$ be a finite set, and consider the product $X=\bbT^2 \times F$ which consists of $\# F$ copies of $\bbT^2$. 
A polygon exchange map on $X$, $T:X\to X$, is a decomposition of $X$ into polygonal pieces $P_1, \ldots, P_n \subset X$,
a choice of elements $\v_1, \ldots, \v_n \in \bbT^2$ and a choice of elements $c_1, \ldots, c_n \in F$ so that the image of the map
$$T(\x, a)=(\x+\v_i,c_i) \textrm{ whenever $(\x,a) \in P_i$}$$
has full area.  We define the subgroup $G=\langle \v_1, \ldots, \v_n \rangle \subset \bbT^2$. 
Fixing a $\y \in \bbT^2$ and an $a \in F$, observe that the orbit of a point $(\y, a)$ is contained in the set
$(\y+G) \times F$. Given any $\y \in \bbT^2$, we define an embedding $\epsilon_{\y}:G \times F \to X$ via
$(\g,a) \mapsto (\y+\g,a)$. For generic $\y$, 
we can pull back the action of $T$ to an action on $G \times F$. We define $\Psi_{\y}:G \times F \to G \times F$ by $\Psi= \epsilon_\y^{-1} \circ T \circ \epsilon_\y$. The {\em arithmetic graph} associated to $T$ and $\y$, $\Gamma(T,\y)$ is the 
graph whose vertices are the points in $G \times F$, and for which an edge is drawn between $(\g,a)$ and $\Psi(\g,a)$
for all $(\g,a) \in G \times F$. Note that the curve of the arithmetic graph through $(\g,a)$
represents the orbit of the point $(\y+\g,a)$ under $T$. So for instance, the point $(\y+\g,a)$ is periodic
if and only if the curve passing through $(\g,a)$ in $\Gamma(T,\y)$ is closed.

The arithmetic graph and similar constructions have been a useful tool for proving results about low complexity dynamical systems which require a detailed understanding orbits. 
For instance, in \cite[Proposition 13]{PLV08} it is shown that the analog of an arithmetic graph for a certain interval exchange consists of only finitely many curves. In \cite{S07} Schwartz coined the term arithmetic graph, and used the arithmetic graph to resolve a long-standing question of Neumann, ``are there outer billiards systems which have unbounded orbits?''
In later work, Schwartz showed that a certain first return maps of outer billiards in polygonal kites is related to polyhedral exchanges by a dynamical compactification construction \cite[The Master Picture Theorem]{S09}. This construction relates the arithmetic graph used in outer billiards with the arithmetic graph of a polygon exchange mentioned above.
The arithmetic graph also played a primary role in \cite{Schwartz10}.

A powerful tool for understanding the dynamical systems mentioned above has been renormalization. For a polygon exchange map $T:X \to X$, for instance, we can construct a new polygon exchange map by considering the return map to a polygon or union of polygons. Repeatedly applying this trick can be useful for proving results about the long term behavior of the system. This is commonly used to answer ergodic theoretic questions about interval exchange maps \cite{MT}.
For a specific polygon exchange map, renormalization has been used to show that the set of points with periodic orbits have full measure \cite{AKT01}. This is similar to our goals here. 

This paper came out of an attempt find a simple example where renormalization can be understood in terms of the arithmetic graph. The polygon exchange maps appearing in \cite{S09} appear complicated, but the associated arithmetic graphs have many nice properties. For instance, the connected components of these arithmetic graphs form a collection of simple curves. Truchet tilings represent a ways to force this behavior independent of constructions involving outer billiards. Moreover, we can obtain many Truchet tilings as instances of arithmetic graphs of polygon exchange maps \cite{H12}. In these cases, the renormalization of tilings described in the introduction corresponds to a renormalization (first return map) of the associated polygon exchange map. 
This paper came out of the realization that, in this particular case, by understanding the action of renormalization on the arithmetic graph, we can generalize the renormalization of a family of polygon exchange maps to a renormalization scheme applicable to a more general class of dynamical systems. These are the systems 
described in Section \ref{sect:truchet shift}. 

Finally, it should be noted that the original motivation for studying Truchet Tilings by Truchet \cite{Truchet} and Smith \cite{Truchet} was aesthetic. These motivations continue today.
See for instance, \cite{LR06}, and \cite{C08}. More general families of curves
associated to tilings have been considered. For instance, \cite{OC99} considers similar curves in Penrose tilings. 

Aside from the above motivations, Truchet tilings have also played role in understanding a variant of the cellular automata known as Langton's Ant \cite{GPST95}.

\subsection{Outline}
In Section \ref{sect:Truchet}, we explain how to think of the space of Truchet tilings as a dynamical system. In Section \ref{sect:truchet shift}, we restrict attention to the case of tilings determined by a pair of elements of the shift space $\Omega_\pm$ as in Equation \ref{eq:tau omega}. We also provide the necessary background on shift spaces
and shift invariant measures here. In Section \ref{sect:drift}, we prove \hyperref[thm:drift]{the Drift Theorem}
and a stronger variant. Section \ref{sect:renormalization} states and proves the renormalization results which apply to many pairs of shift-invariant measures. This culminates in Section \ref{sect:renormalization2}, which explains
the renormalization process and constructs the cocycle alluded to in the introduction. Finally, Section \ref{sect:random} develops these renormalization theorems in the context of the measures relevant to Theorem \ref{thm:random tilings}. Section \ref{sect:evaluating the limit} does the necessary calculation to prove this theorem.

\section{Dynamics on Truchet tilings}
\name{sect:Truchet}
Consider the unit square centered at the origin with horizontal and vertical sides.
An {\em inward normal} to the square is a unit vector which is based at a midpoint of an edge and is pointed into the square. See below
\begin{center}
\includegraphics[height=0.75in]{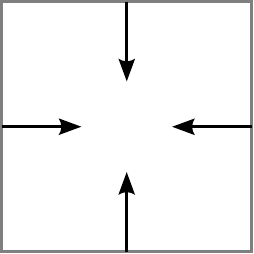}
\end{center}
We do not keep track of the location at which the normal is places, so the four inward normals are just the vectors $(1,0)$,
$(0,1),$
$(-1,0)$ and
$(0,-1)$. We use $N$ to denote the collection of inward normals.


Let $\sT$ be the collection of maps $\Z^2 \to \{\pm 1\}$. 
We will define a dynamical system on $\sT \times N$. Choose $(\tau, \v) \in \sT \times N$. The inward normal $\v \in N$ is a vector based at a midpoint of the square at the origin pointed inward. The Truchet tiling $[\tau]$ determined by $\tau$ places the tile $T_{\tau(0,0)}$ at the origin.
We drag the vector inward along an arc of this tile keeping the vector tangent to the arc. After a quarter turn, we end up as a vector pointed out of the square centered at the origin.
So, the vector points into a square adjacent to the square at the origin. We translate the whole tiling so that this new square becomes centered at the origin.

Formally, this is the dynamical system $\Phi_0:\sT \times N \to \sT \times N$ given by 
\begin{equation}
\name{eq:Phi}
\Phi_0 \big(\tau, (a,b)\big)=\big(\tau \circ S_{s(b,a)}, s(b,a)\big),
\end{equation}
where $s=\tau(0,0) \in \{\pm 1\}$ and $S_{s(b,a)}$ is the translation of the plane by the vector
$s(b,a)$. 

\section{Truchet tiling spaces from shift spaces}
\name{sect:truchet shift}
In this paper we will concentrate on Truchet tilings which arise from a pair of bi-infinite sequences of elements of the set $\{\pm 1\}$. 
As in the introduction, we will use notation from the world of shift spaces to denote these bi-infinite sequences. Namely, an element $\omega \in \Omega_\pm$ is a bi-infinite sequence
of elements of $\{\pm 1\}$. For $n \in \Z$, we use $\omega_n$ to denote the $n$-th element of the sequence $\omega$. 

Given $\omega, \omega'\in \Omega_\pm$, we obtain a function 
$\tau_{\omega, \omega'}:\Z^2 \to \{\pm 1\}$ as in equation \ref{eq:tau omega} of the introduction. 
In this paper, we will be interested in studying the dynamics of $\Phi_0$ on the collection of all $\tau_{\omega, \omega'}$. 
This collection of tilings is $\Phi_0$-invariant and admits a natural renormalization procedure as we explain in section \ref{sect:renormalization}. In this section, we reveal some of the more basic structure of the map $\Phi_0$
restricted to these types of tilings.

\subsection{Background on shift spaces}
\name{sect:topology}
We briefly recall some basic facts about two-sided shift spaces here. For further background see \cite{LindMarcus}, for instance.

Let $\sA$ be a finite set called an {\em alphabet}. The set $\sA^\Z$ is called the {\em full (two-sided) shift} on $\sA$. 

For integers $m<n$, let $f:\{m,m+1, \ldots, n\} \to \sA$ be an arbitrary function. 
The {\em cylinder set} determined by $f$ is the set
\begin{equation}
\name{eq:cylinder set}
\sC(f)=\{\omega \in \sA^\Z~:~\omega_i=f(i) \textrm{ for all $i=m, \ldots, n$}.\}
\end{equation}
We equip $\sA^\Z$ with the topology generated by the cylinder sets. This is the coarsest topology which makes each cylinder set open.
Observe that the cylinder sets are also closed. An equivalent description of this topology is obtained by considering elements of $\sA^\Z$
as functions $\Z \to \sA$. From this point of view, this is the topology of pointwise convergence on compact subsets of $\Z$, where
$\sA$ is given the discrete topology.
With this topology, the set $\sA^\Z$ is homeomorphic to a Cantor set.

The {\em shift map} $\sigma:\sA^\Z \to \sA^\Z$ is the homeomorphism of the full shift space defined by
\begin{equation}
\name{eq:shift}
[\sigma(\omega)]_n=\omega_{n+1}.
\end{equation}
A {\em shift space} $\Omega$ is a closed, shift-invariant subset of a full shift space. We endow $\Omega$ with the subspace topology.

A {\em shift-invariant measure} on a shift space $\Omega$ is a Borel measure $\mu$ satisfying 
$$\mu \circ \sigma^{-1}(A)=\mu(A) \quad \textrm{for all Borel subsets $A \subset \Omega$}.$$
Full shift spaces admit a plethora of shift-invariant probability measures.

\subsection{Tiling spaces from shift spaces}
\name{sect:tilings from shift spaces}
As before, we let $\Omega_\pm=\{\pm 1\}^\Z$ denote the full shift space on the alphabet $\{\pm 1\}$.  Given a pair of elements $\omega, \omega' \in  \Omega_\pm$, we 
obtain a map $\tau_{\omega,\omega'}: \Z^2 \to \{\pm 1\}$ as in equation \ref{eq:tau omega}. This function in turn determines a tiling $\tiling{\tau_{\omega, \omega'}}$ as described in the introduction.
Note that the map $(\omega, \omega') \to \tau_{\omega, \omega'}$ is two-to-one, with $\tau_{\omega,\omega'}=\tau_{-\omega,-\omega'}$, where
$-\omega$ denotes the element of $\Omega_\pm$ given by $(-\omega)_n=-\omega_n$. 

The collection $\big \{\tau_{\omega, \omega'}~:~\omega, \omega' \in \Omega_\pm\big\}$ is easily seen to be translation invariant, and therefore
the set $\big \{\tau_{\omega, \omega'}~:~\omega, \omega' \in \Omega_\pm\big\} \times N$ is invariant under $\Phi_0$. There is a natural lift of the action
of $\Phi_0$ on this set to an action on the set $X=\Omega_\pm \times \Omega_\pm \times N$ given by 
\begin{equation}
\name{eq:Phi2}
\Phi\big(\omega, \omega', (a,b)\big)=\big(\sigma^{sb} (\omega), \sigma^{sa} (\omega'), s(b,a)\big),
\end{equation}
with $s=\omega_0 \omega'_0$. (Here, $\sigma^{-1}$ denotes the inverse of the shift map defined by $\big(\sigma^{-1}(\omega)\big)_n=\omega_{n-1}$.)
We call $\Phi$ a lift of $\Phi_0$ because if $\Phi(\omega,\omega',\v)=(\eta, \eta',\v')$ then $\Phi_0(\tau_{\omega,\omega'},\v)=(\tau_{\eta,\eta'},\v')$.
Observe that the inverse of $\Phi$ is given by 
\begin{equation}
\Phi^{-1}\big(\omega, \omega', (a,b)\big)=\big(\sigma^{-a} (\omega), \sigma^{-b} (\omega'), \omega_{-a} \omega'_{-b}(b,a)\big).
\end{equation}
We now make some preliminary observations about $\Phi$. 

\begin{proposition}
Let $\Omega, \Omega' \subset \Omega_\pm$ be shift spaces. The set $\Omega \times \Omega' \times N$ is a closed $\Phi$-invariant subset of $\Omega_\pm \times \Omega_\pm \times N$. 
\end{proposition}
The above proposition trivially follows from the definitions. 

\begin{proposition}
\name{prop:product measures}
Suppose $\mu$ and $\mu'$ are shift invariant probability measures on $\Omega$ and $\Omega'$, respectively. Let $\mu_N$ be the discrete probability measure
on $N$ so that $\mu_N(\{\v\})=\frac{1}{4}$ for each $\v \in N$. Then $\mu \times \mu' \times \mu_N$ is a $\Phi$-invariant probability measure on $\Omega \times \Omega' \times N$.
\end{proposition}
\begin{proof}
It is sufficient to show $\Phi$-invariance
of $\nu=\mu \times \mu' \times \mu_N$ on sets of the form $A \times B \times \{\v\}$, where $\v \in N$ and $A \subset \Omega$ and $B \subset \Omega'$ are Borel sets
chosen so that the product $s=\omega_0 \omega'_0$ is independent of the choice of $\omega \in A$ and $\omega' \in B$. 
Then by definition of $\Phi$, shift invariance of the measures $\mu$ and $\mu'$, and the permutation invariance of $\mu_N$,
$$\nu \circ \Phi(A \times B \times \{(a,b)\})=\nu\big(\sigma^{sb} (A) \times \sigma^{sa} (B) \times \{s(b,a)\}\big)=
\nu(A \times B \times \{(a,b)\}).$$
\end{proof}

\subsection{Periodic orbits}
\name{sect:periodic orbits}
Suppose $(\omega, \omega',\v) \in X$ satisfies $\Phi^n(\omega, \omega',\v)=(\omega, \omega',\v)$. We say $(\omega, \omega',\v)$ has an {\em stable periodic orbit of period $n$} if $n$ is the smallest number for which there are open neighborhoods $U$ and $U'$ of $\omega$ and $\omega'$ respectively for which each $(\eta, \eta', \v) \in U \times U' \times \{\v\}$ satisfies $\Phi^n(\eta, \eta',\v)=(\eta, \eta',\v)$. The following proposition characterizes the points with stable periodic orbits.

\begin{proposition}[Stability Proposition]
\name{prop:stable}
Let $(\omega, \eta, \v) \in X$, and use $(\omega^k, \eta^k, \v^k)$ to denote $\Phi^k(\omega, \omega', \v)$.
The following statements hold.
\begin{enumerate}
\item If $(\omega, \omega',\v) \in X$ has a stable periodic orbit of period $n$, then it is also (least) period-$n$ under $\Phi$ in the usual sense.
\item $(\omega, \omega',\v) \in X$ has a stable periodic orbit if and only if the curve of the tiling $[\tau_{\omega, \omega'}]$ passing through the normal $\v$ to the square centered at the origin is closed. In this case, the period of $(\omega, \omega',\v)$ is the number of squares the associated closed curve of the tiling intersects, counting multiplicities.
\item $(\omega, \omega',\v) \in X$ has a stable periodic orbit of period $n$ if and only if 
$n$ is the smallest positive integer for which both $\Phi(\omega, \omega',\v)=(\omega, \omega',\v)$ and 
$\sum_{i=0}^{n-1} \v^i=(0,0)$. 
\end{enumerate}
\end{proposition}
The multiplicities mentioned in the proposition deal with the fact that curves may (a priori) intersect the same square twice. (In fact, no curve intersects a square twice. This follows from Lemma \ref{lem:M}, below.)

\begin{proof}
For $n \geq 1$ let 
$a_n=\sum_{k=1}^n \v^k_1$ and $b_n=\sum_{k=1}^n \v^k_2$.
By induction using equation \ref{eq:Phi2}, we observe that $\omega^n=\sigma^{a_n}(\omega)$ and $\eta^n=\sigma^{b_n}(\omega')$.
Now suppose $(\omega^k, \eta^k, \v^k)$ is period $n$. Note that $(a_n, b_n)$ equals the vector sum in statement (3).
We see $\sigma^{a_n}(\omega)=\omega$ and $\sigma^{b_n}(\omega')=\omega'$. 
Therefore if $a_n \neq 0$, we see that $\omega$ must be period-$a_n$ under the shift map. This is not an open condition, so 
  $(\omega, \omega',\v)$ cannot have a stable periodic orbit of period $n$. This is similarly true if $b \neq 0$. Extending this argument, we observe
that $a_{kn}=k a_n$ and $b_{kn}=k b_n$ for integers $k>1$. Therefore, $(\omega, \omega',\v)$ cannot have a stable periodic orbit of any period.
Conversely, if both $a_n=0$ and $b_n=0$, then we let $a_+=\max \{a_1, \ldots, a_n\}$ and $a_-=\min \{a_1, \ldots, a_n\}$. Define $b_+$ and $b_-$ similarly. 
Then consider the open sets
$$U=\set{\alpha \in \Omega_\pm}{$\alpha_k=\omega_k$ for all $k$ with $a_- \leq k \leq a_+$}, \textrm{and}$$
$$U'=\set{\beta \in \Omega_\pm}{$\beta_k=\omega'_k$ for all $k$ with $b_- \leq k \leq b_+$}.$$
Then observe that for $\alpha \in U$ and $\beta \in U'$ we have $\Phi^n(\alpha, \beta,\v)=(\alpha, \beta,\v)$.
Therefore, $(\omega, \omega',\v) \in X$ has a stable periodic orbit of period $n$.  Finally, observe that the condition that $a_n=0$ and $b_n=0$
is equivalent to the statement that the curve of the tiling $\tiling{\tau_{\omega, \omega'}}$ passing through $\v$ is closed. 
\end{proof}

\begin{remark}
Not all periodic orbits are stable. When $\omega_n=1$ and $\omega'_n=1$ for all $n \in \Z$, we have
$\Phi^2(\omega, \omega',\v)=(\omega, \omega',\v)$ for all $\v$, but $(\omega, \omega',\v)$ is not an $n$-stable 
periodic orbit for any $n$.
\end{remark}

\begin{corollary}
\name{cor:restatement}
Let $\mu$ and $\mu'$ be shift-invariant measures on $\Omega_\pm$. Let $P_n \subset X$ be 
the set of all $(\omega, \omega',\v)$ with stable periodic orbits of period $n$. Fix an edge $e$ of
the tiling of the plane by squares centered at the integer points as in the theorem of the introduction.
Then, $\mu \times \mu' \times \mu_N(P_n)$ is equal to the $\mu \times \mu'$ measure of those
$(\omega, \omega')$ so that the curve of the tiling $\tiling{\tau_{\omega, \omega'}}$ through $e$
is closed and visits $n$ squares (counting multiplicities).
\end{corollary}

The proof follows from \hyperref[prop:stable]{the Stability Proposition} together with the observation that
both quantities are translation invariant. The fact that the horizontal or vertical orientation of $e$ is irrelevant follows from the fact that curves
of the tiling alternate intersecting horizontal and vertical edges. We omit a detailed proof of this corollary.

\subsection{An invariant function and drift}
\name{sect:drift}
In this section, we prove \hyperref[thm:drift]{the Drift Theorem}. Ideas for this result come from the drift vector of an interval exchange transformation. See \cite{PLV08}, for instance.

The first observation
of this section is that there is a simple $\Phi$ invariant function on
$X=\Omega \times \Omega' \times N$.

\begin{lemma}[Invariant function]
\name{lem:M}
Let $M\big(\omega, \omega', (a,b)\big)=b \omega_0+a \omega_0'$. This is a $\Phi$-invariant
function from $X$ to $\{\pm 1\}$.
\end{lemma}
\begin{proof}[Sketch of proof]
We partition the space $\Omega_\pm \times \Omega_\pm \times N$ into $16$ subsets $G\big(s,s',(a,b)\big)$ according to choices of $s,s' \in \{\pm 1\}$ and $(a,b) \in N$. These groups are defined 
$$G\big(s,s',(a,b)\big)=\big\{\big(\omega, \omega', (a,b)\big) \in \Omega_\pm \times \Omega_\pm \times N~:~
\textrm{$\omega_0=s$ and $\omega'_0=s'$}\}.$$
Write $\sG$ for the set of these 16 subsets. Let $\sim$ be the strongest equivalence relation on $\sG$ for which 
$G_1 \sim G_2$ whenever $\Phi(G_1)$ intersects $G_2$. The equivalence classes can be computed by drawing the graph where
the nodes are elements of $\sG$ and the arrows are drawn from $G_1$ to $G_2$ whenever $\Phi(G_1)$ intersects $G_2$;
the equivalence classes are then the connected components of this graph.
One of the two maximal equivalence classes is shown below. 
\begin{center}
\begin{xy}
(10,20)*+{G\big(1,1,(1,0)\big)}="a"; 
(50,20)*+{G\big(1,-1,(0,1)\big)}="b";
(90,20)*+{G\big(-1,-1,(-1,0)\big)}="c";
(130,20)*+{G\big(-1,1,(0,1)\big)}="d";
(10,0)*+{G\big(1,1,(0,1)\big)}="ap"; 
(50,0)*+{G\big(1,-1,(-1,0)\big)}="bp";
(90,0)*+{G\big(-1,-1,(0,-1)\big)}="cp";
(130,0)*+{G\big(-1,1,(1,0)\big)}="dp";
{\ar@{->} "a";"b"};
{\ar@{->} "b";"c"};
{\ar@{->} "c";"d"};
{\ar@{<-} "ap";"bp"};
{\ar@{<-} "bp";"cp"};
{\ar@{<-} "cp";"dp"};
{\ar@{<->} "a";"ap"};
{\ar@{<->} "b";"bp"};
{\ar@{<->} "c";"cp"};
{\ar@{<->} "d";"dp"};
{\ar@{->}@/_{1.5pc}/ "d";"a"};%
{\ar@{->}@/_{1.5pc}/ "ap";"dp"};%
\end{xy}
\end{center}
Note that $M\equiv1$ on this equivalence class. The function $M$ takes the value $-1$ on the eight remaining
subsets.
\end{proof}

The following is a restatement of the Drift Theorem in the introduction. Equivalence follows from the Corollary \ref{cor:restatement}.
\begin{theorem}[Restated Drift Theorem]
Suppose $\mu$ and $\mu'$ are shift-invariant probability measure on $\Omega_\pm$ satisfying
$$p=\int_{\Omega_\pm} \omega_0~d\mu(\omega) \and q=\int_{\Omega_\pm} \omega'_0~d\mu'(\omega').$$
Then the $\mu \times \mu' \times \mu_N$ measure of the set of all $(\omega, \omega',\v)$ 
without stable periodic orbits is at least $\max \{|p|, |q|\}$. 
\end{theorem}
\begin{proof}
Let $X_s=M^{-1}(\{s\})$ for $s \in \{\pm 1\}$. We would like to compute the integral
$$I=\int_{X_s} (a,b)~d(\mu \times \mu' \times \mu_N),$$
with the integral taken over all $\big(\omega, \omega',(a,b)\big) \in X_s.$
Let $X_s(a,b)$ denote those $(\omega, \omega',\v) \in X_s$ with $\v=(a,b)$. Then,
$$I=\frac{1}{4} \sum_{(a,b) \in N} \int_{X_s(a,b)} (a,b)~d(\mu \times \mu'),$$
with the integral take over all pairs $(\omega,\omega')$ with $(a,b)$ fixed by the sum.
The $\frac{1}{4}$ appears because of the removal of $\mu_N$. Consider the case
$(a,b)=(1,0)$. Note that $M=\big(\omega,\omega',(1,0)\big)=\omega'_0$, so that for this term
$$\int_{X_s(1,0)} (1,0)~d(\mu \times \mu')=\int_{X_s(1,0)} (s \omega'_0,0) ~d(\mu \times \mu')=s(q,0).$$
Similarly, in the case $(a,b)=(-1,0)$, we have $M=\big(\omega,\omega',(-1,0)\big)=-\omega'_0$, so again
$$\int_{X_s(-1,0)} (-1,0)~d(\mu \times \mu')=\int_{X_s(s \omega'_0,0)} (-\omega'_0,0) ~d(\mu \times \mu')=s(q,0).$$
Similar analysis holds for the cases $(a,b)=(\pm 0,1)$ and show that the total integrals is given by 
$I=\frac{s}{2}(q,p)$. 

Let $P_s$ denote the set of all $\big(\omega, \omega',(a,b)\big) \in X_s$ which have stable periodic parameters.
This set is $\Phi$-invariant, and the proposition above guarantees that 
$$\int_{P_s} (a,b)~d(\mu \times \mu' \times \mu_N)=0.$$
Also note that for any $\big(\omega, \omega',(a,b)\big)$ that if $(a',b')$ is the $N$ component 
of $\Phi(\big(\omega, \omega',(a,b)\big)$ then $(a,b)+(a',b')$ is one of the four vectors of the form $(\pm 1, \pm 1)$.
Therefore, for any $\Phi$-invariant set $A \subset \Omega_\pm \times \Omega_\pm \times N$ with
$$\int_{A} (a,b)~d(\mu \times \mu' \times \mu_N)=(x,y)$$
we have $|x| \leq \frac{1}{2} \mu \times \mu' \times \mu_N(A)$ and $|y| \leq \frac{1}{2} \mu \times \mu' \times \mu_N(A)$. Applying this to the invariant set $X_s \smallsetminus P_s$, we see 
$$\frac{s}{2}(q,p)=I=\int_{X_s \smallsetminus P_s} (a,b)~d(\mu \times \mu' \times \mu_N)$$
so that $\mu \times \mu' \times \mu_N(X_s \smallsetminus P_s) \leq \max \{|p|, |q|\}$, as desired.
\end{proof}

We get a stronger result using the ergodic decomposition. Let us briefly recall the statement in this context. Let $\sM$ denote the collection of shift-invariant probability measures on $\Omega_\pm$, and $\sE \subset \sM$ denote those measures which are ergodic. For any shift invariant probability measure $\mu$, there is unique probability measure $\lambda$ defined on $\sM$ so that $\lambda(\sE)=1$ and for all continuous $f:\Omega_\pm \to \R$ we have
$$\int_{\Omega_{\pm}} f~d\mu=\int_\sE \int_{\Omega_\pm} f(x)~d\nu(x)~d\lambda(\nu).$$

Now let $\chi:X \to \R$ be the characteristic function on the set of all $(\omega, \omega',\v)$ without stable periodic orbits. 
Let $\mu$ be a shift-invariant probability measure as above. And $\lambda$ be the measure obtained from the ergodic decomposition. Then, applying the Drift Theorem to the ergodic measures yields
$$\begin{narrowarray}{0pt}{rcl}
\int_X \chi(x)~d (\mu \times \mu' \times \mu_N)(x) & = & 
\int_{\Omega_\pm} \int_{N} \int_{\Omega_\pm} \chi(\omega, \omega', \v)~d \mu'(\omega')~d\mu_N(\v)~d\mu(\omega) \\
& = &
\int_\sE \int_{\Omega_\pm} \int_{N} \int_{\Omega_\pm} \chi(\omega, \omega', \v)~d \mu'(\omega')~d\mu_N(\v)~d\nu(\omega)~d\lambda(\nu) \\
& \geq & \int_\sE \big|\int_{\Omega_\pm} \omega_0~d \nu(\omega)|~d\lambda(\nu).
\end{narrowarray}$$
The following is what is weaker than the above argument gives.

\begin{corollary}
Let $\lambda$ and $\lambda'$ be the measures obtained from the ergodic decomposition applied to $\mu$ and $\mu'$, respectively. If the $\mu \times \mu' \times \mu_N$ measure of the set of all $(\omega, \omega',\v)$ 
without stable periodic orbits is zero, then for $\lambda$-a.e. (and $\lambda'$-a.e) $\nu \in \sE$
we have $\int_{\Omega_\pm} \omega_0~d \nu(\omega)=0$. 
\end{corollary}

\section{Renormalization of Truchet tilings}
\name{sect:renormalization}
In this section, we will describe the renormalization procedure for the dynamical system
$\Phi:X \to X$. Informally, this procedure can be described in terms of tilings as in the introduction. Given a tiling  $[\tau_{\omega, \omega'}]$,
we renormalize to obtain a new tiling $[\tau_{\eta, \eta'}]$ with $\eta=c(\omega)$ and $\eta'=c(\omega')$. 
The function $c$ is called a collapsing map and is defined in Subsection \ref{sect:collapsing}. 
The renormalized tiling is obtained from the original by collapsing some rows and columns of tiles to lines. This is explained in the following subsection.

\subsection{Notation for words}
A {\em word} in $\{\pm 1\}$ is an element $w$ of a set $\{\pm 1\}^{\{1,\ldots,n\}}$ for some $n=\ell(w) \geq 0$, called the {\em length} of $w$. 
We use $\sW$ to denote the collection of all words. 
We write $w=w_1 \ldots w_n$ with $w_i \in \{\pm 1\}$ to denote a word, and 
$\emptyset$ to denote the unique word of length $0$. 
To simplify notation of the elements in $\{\pm 1\}$, we use $+$ to denote $1$ and $-$ to denote $-1$. So the word 
$w$ where $w_1=1$ and $w_2=-1$ can be simply written as $w=+-$. 

Adjacency indicates the {\em concatenation} of words; if $w,w' \in \sW$ then 
$$ww'=w_1 \ldots w_{\ell(w)} w'_1 \ldots w'_{\ell(w')}.$$ 
For an integer $n \geq 0$, the expression $w^n$ indicates the n-fold concatenation $ww \ldots w$, with $w$ appearing $n$ times. 

If $a \leq b \leq c$ and $w=w_a \ldots w_c$ with $w_i \in \{\pm 1\}$, then we can consider the function
$f:\{a-b, \ldots, c-b\} \to \{\pm 1\}$ given by $f(i-b)=w_i$. By equation \ref{eq:cylinder set}, this determines a cylinder set $\sC(f)$, which we also denote by 
$\sC(w_a \ldots \widehat w_b \ldots w_c)$, with the hat indicating that $w_b$ represents the value of the zeroth entry of the words in the cylinder set.

\subsection{The collapsing map on shift spaces}
\name{sect:collapsing}
The idea of the collapsing function $c$ mentioned at the beginning of this section is to removed any substrings of the form $-+$ and then slide the remaining
entries together. For example,
$$c(\ldots \underline{-+}+--\underline{-+}\underline{-+}\widehat+-\underline{-+}++\ldots)=\ldots +--\widehat+-++\ldots,$$
where underlined entries have been removed. There are two potential reasons why $c(\omega)$ may not be well defined.
First, the zeroth entry might be removed by this process. Second, the remaining list may not be bi-infinite.

Formally, we define $S \subset \Omega_\pm$ to be the union of two cylinder sets
$S=\sC(\widehat{-}+) \cup \sC(-\widehat{+}).$
Given $\omega \in \Omega_\pm$, we set $K(\omega) \subset \Z$ to be
\begin{equation}
\name{eq:K}
K(\omega)=\{k \in \Z~:~\sigma^k(\omega) \not \in S\}.
\end{equation}
We say that $\omega \in \Omega_\pm$ is {\em unbounded-collapsible} if $K(\omega)$ is unbounded in both directions,
{\em zero-collapsible} if $0 \in K(\omega)$, and {\em collapsible} if it is both unbounded- and zero-collapsible.
We use $C$ to denote the collection of collapsible elements of $\Omega_\pm$.
If $\omega \in C$, then there is a unique order preserving indexing $\Z \to K(\omega)$ denoted $i \mapsto k_i$ so that 
$K(\omega)=\{k_i~:~i \in \Z\}$ and $k_0=0$. We use $c(\omega) \in  \Omega_\pm$ to denote the {\em collapse of $\omega$},
which we define by $c(\omega)_i=\omega_{k_i}$. So, the {\em collapsing map} is a map $c:C \to \Omega_\pm$. 

In the remainder of this subsection, we investigate properties of this map.
\begin{proposition}
The collapsing map is surjective.
\end{proposition}
To prove this proposition, we explicitly construct the inverse images. For this, we define a process we call insertion.
An insertion rule is determined by a function $f:\Z \to \sW$. Let $\omega \in \Omega_\pm$. 
From $f$, we determine a strictly increasing bi-infinite sequence of integers $\langle m_i \rangle_{i \in \Z}$ inductively according to the following two rules.
\begin{itemize}
\item $m_0=0$. 
\item For all $i \in \Z$, then $m_{i+1}-m_{i}=1+\ell \circ f(i)$.
\end{itemize}
The {\em insertion function determined by $f$}, $\sI_f:\Omega^\pm \to \Omega^\pm$ evaluated at $\eta$ is given by
the following rules.
\begin{itemize}
\item $\sI_f(\eta)_k=\eta_i$ if $k=m_i$ for some $i \in \Z$. 
\item $\sI_f(\eta)_k=f(i)_{k-m_i}$ if $m_{i}<k<m_{i+1}$ for some $i \in \Z$. 
\end{itemize}

\begin{proposition}
\name{prop:inverses of collapsing map}
Let $\sW_{-+}=\{(-+)^n~:~n \geq 0\}$. For all $\eta \in \Omega_\pm$, we have
$$c^{-1}(\eta)=\{\sI_f(\eta)~:~\textrm{$f(i) \in \sW_{-+}$ for all $i$ and $f(i) \neq \emptyset$ whenever $\eta_i=-1$ and $\eta_{i+1}=1$}\}.$$
\end{proposition}
\begin{proof}
Suppose $c(\omega)=\eta$. We will show that $\omega$ is an element of the set on the right hand side of the equation. 
Consider the set $K(\omega)$ indexed by $i \mapsto k_i$ as in the definition of the collapsing map. By definition of $c$, 
for all $i$ we have $\eta_{k_i}=\omega_i$, and $\eta_{k_i+1} \ldots \eta_{k_{i+1}-1} \in \sW_{-+}$. This proves that $\eta=\sI_f(\omega)$ with $f(i) \in \sW_{-+}$ for all $i$.  
Finally, by definition of $K(\eta)$, when $k_{i+1}=k_i+1$ we must have $k_{i} \neq -1$ or $k_{i+1} \neq 1$. This is equivalent to the statement
that $f(i) \neq \emptyset$ whenever $\omega_i=-1$ and $\omega_{i+1}=1$. 

Conversely, suppose $\omega=\sI_f(\eta)$ with $f$ as above. Define $K(\omega)$. 
Also define $m_i$ as in the definition of insertion function applied to $\eta$. 
Let $M=\{m_i~:~i \in \Z\}$. We will show that $M=K(\omega)$. Then it follows from the definitions of $c$ and $\sI_f$ that $c(\omega)=\eta$ as desired.
It is clearly true that 
$M \subset K(\eta)$, because every inserted word is of the form $(-+)^n$ for some $n \geq 0$. 
Now we show that $K(\omega) \subset M$. Suppose $k \in K(\omega) \smallsetminus M$.
Since $k \not \in M$, we note that $m_i<k<m_{i+1}$ for some $i \in \Z$. 
Therefore by the definition of insertion function and the fact we only insert words in $\sW_{-+}$, we have that the word
$$\omega_{m_i+1} \omega_{m_i+2} \ldots \omega_{m_{i+1}-1} =(-+)^n.$$
for $n=\frac{m_{i+1}-m_{i}-1}{2} >0$. Observe by definition of $K(\omega)$ that
$\{m_i+1, m_i+2, \ldots, m_{i+1}-1\} \not \subset K(\omega)$ and therefore $k \not \in K(\omega)$, which is a contradiction.
\end{proof}

Suppose $a,b \in \Z$ with $a \leq 0 \leq b$ and $w_{a}, \ldots, w_{b} \in \{\pm 1\}$. Recall, we use $\sC(w_a \ldots \widehat{w_0} \ldots w_b)$
to denote the cylinder set consisting of those $\omega$ for which $\omega_k=w_k$ for all $k$ satisfying $a\leq k\leq b$. 
\begin{corollary}[Preimages of cylinder sets]
\name{cor:preimages of cylinders}
Let $C_u$ denote the set of unbounded collapsible elements of $\Omega_\pm$.
Suppose $a \leq 0 \leq b$ and $w_{a}, \ldots, w_{b} \in \{\pm 1\}$. Then $c^{-1}\big(\sC(w_a \ldots \widehat{w_0} \ldots w_b)\big)$
is given by
$$C_{u} \cap \bigcup_{n_a, \ldots, n_{b-1}} \sC(s_- w_a (-+)^{n_a} w_{a+1} (-+)^{n_{a+1}} \ldots \widehat{w_0} \ldots (-+)^{n_{b-2}} w_{b-1} (-+)^{n_{b-1}} w_b s_+),$$
where the union is taken over all choices of integers $n_k$ such that $n_k \geq 1$ if $w_k=-1$ and $w_k{+1}=1$, and $n_k \geq 0$ otherwise. The word $s_-=+$ if $w_a=1$ and $s_-=\emptyset$ otherwise. The word $s_+=-$ if $w_b=-1$ and $s_+=\emptyset$ otherwise.
\end{corollary}

\begin{corollary}
The collapsing map is continuous.
\end{corollary}
\begin{proof}
By the above corollary, the inverse image of a cylinder set is a union of cylinder sets intersected with $C$ and therefore open in $C$. Since the cylinder sets form a basis for the topology, the inverse image of any open set is open in $C$. Therefore, $c$ is continuous.
\end{proof}

\begin{proposition}
\name{prop:omegaalt}
If $\Omega \subset \Omega_\pm$ is a shift space
and the alternating element $\omegaalt$ defined by 
$\omegaalt_n=(-1)^n$ is not an element of $\Omega$,
then every element of $\Omega$ is unbounded-collapsible.
\end{proposition}
\begin{proof}
Suppose $\omega \in \Omega$ is not unbounded-collapsible. Then, $\omega_{n+1}=-\omega_n$ for all $n>N$ or all $n<N$ for some $N \in \Z$.
Therefore, $\omegaalt$ can be obtained as a limit of shifts of $\omega$.
\end{proof}

We have the following analog for measures.
\begin{proposition}
\name{prop:omegaalt2}
Suppose $\mu$ is a shift-invariant probability measure on $\Omega_\pm$, and that
$$\mu(\set{\omega \in \Omega_\pm}{$\omega$ is not unbounded-collapsible})>0.$$
Then $\mu(\{\omegaalt\})>0$. 
\end{proposition}

\begin{lemma}
Suppose $A \subset \Omega_{\pm}$ is shift-invariant, then so is $c(A \cap C)$. If $\Omega \subset \Omega_\pm$ is a shift space and
$\omegaalt \not \in \Omega$, then $c(\Omega \cap C)$ is a shift space.
\end{lemma}
\begin{proof}
To prove the first statement we will show that if $\eta=c(\omega)$ with $\omega \in A$ then $\sigma^{\pm 1} (\eta) \in c(A \cap C)$. 
Consider the elements $k_{\pm 1} \in K(\omega) \subset \Z$, as defined in the definition of the collapsing map. We have
$\sigma^{\pm 1}(\eta)=c \circ \sigma^{k_{\pm 1}}(\omega)$. Since $\omegaalt \not \in \Omega$, all elements in $\Omega \smallsetminus C$
are not zero-collapsible. So, $\Omega \cap C$ is the intersection of $\Omega$ with two cylinder sets. So, $\Omega \cap C$ is closed.
Therefore $c(\Omega \cap C)$ is closed by the continuity of $c$ and compactness of $\Omega \cap C$.
\end{proof}

\begin{proposition}
\name{prop:collapsing measures}
Let $\mu$ be a shift invariant measure on a shift space $\Omega \subset \Omega_{\pm}$. Then,
$\mu \circ c^{-1}$ is a shift invariant measure on $c(\Omega)$. 
\end{proposition}
\begin{proof}
The content of this proposition is that $\mu \circ c^{-1}$ is shift invariant.
Let $B \subset c(\Omega)$ be a Borel set and $A=c^{-1}(B)$. Recall the definition of $K(\omega) \subset \Z$ used in the collapsing map. Let $k_1(\omega)$ denote the smallest positive entry of $K(\omega)$. For $j\geq 1$ let 
$A_j=\{\omega \in A~:~k_1(\omega)=j\}.$ 
Then we have $\mu \circ c^{-1}(B)=\sum_{j=1}^\infty \mu(A_j)$. 
For $\omega \in A_j$ we have $c \circ \sigma^j(\omega)=\sigma \circ c(\omega)$. 
Observe that $c^{-1} \circ \sigma(B)=\bigsqcup_{j=1}^\infty \sigma^j(A_j)$. 
Therefore by shift invariance of $\mu$, 
$$\mu \circ c^{-1} \circ \sigma(B)=\sum_{j=1}^\infty \mu \circ \sigma^j(A_j)=\sum_{j=1}^\infty \mu(A_j)=\mu(B).$$
\end{proof}

\subsection{Renormalization Theorems}
\name{sect:renormalization2}
In this section we state our most general renormalization results for the map $\Phi:X \to X$ where $X=\Omega_\pm \times \Omega_\pm \times N$ as in Section \ref{sect:tilings from shift spaces}. The main results are very combinatorial, so we prove them in the next subsection. However, corollaries we state are proved to follow from the main results. 

Define $\sR_1 \subset X$ to be the set of ``once renormalizable'' elements of $X$:
\begin{equation}
\name{eq:R1}
\sR_1=\set{(\omega, \omega', \v)}{both $\omega$ and $\omega'$ are collapsible}=C \times C \times N.
\end{equation}
The renormalization mentioned is the map
\begin{equation}
\name{eq:rho}
\rho:\sR_1 \to X \quad \textrm{given by $\rho(\omega, \omega', \v)=\big(c(\omega), c(\omega'), \v\big)$}.
\end{equation}
The manner in which $\rho$ renormalizes the map $\Phi$ is described by the theorem below.

Before stating the theorem, we define some important subsets of $X$:
$$P_4=\{\textrm{$(\omega, \omega', \v)\in X$ which have a stable periodic orbit of period $4$}\}.$$
$$\NUC=\set{(\omega, \omega', \v)\in X}{either $\omega$ or $\omega'$ is not unbounded-collapsible}.$$
$$NS=\{\textrm{$(\omega, \omega', \v)\in X$ without a stable periodic orbit}\}.$$
The points in $P_4$ correspond to loops in a tiling of smallest possible size. See Proposition \ref{prop:stable}.
The points of $\NUC$ fail to be renormalizable ($\sR_1 \cap \NUC=\emptyset$) in the most dramatic way:
if $(\omega, \omega', \v) \in \NUC$, then $\big(\sigma^m(\omega), \sigma^n(\omega'),\v\big) \in \NUC$
for all $m,n \in \Z$. But, we think of $\NUC$ as a very small set. 
See Propositions \ref{prop:omegaalt} and \ref{prop:omegaalt2}.

\begin{theorem}[Renormalization]
\name{thm:renormalization}
Let $\omega, \omega' \in \Omega_\pm$ and $\v \in N$. Then the following statements hold. 
\begin{enumerate}
\item $P_4 \cap \sR_1= \emptyset$ (stable periodic orbits of period $4$ are not renormalizable.)
\item If $(\omega,\omega',\v) \in X \sm (P_4 \cup \NUC)$, then there are integers $m<0$ and $n>0$ for which $\Phi^m(\omega, \omega', \v), \Phi^n(\omega, \omega', \v) \in \sR_1$. 
In particular, the first return map $\Phi_R:\sR_1 \to \sR_1$ of $\Phi$ to $\sR_1$ is well defined and invertible.
\item If $(\omega, \omega', \v) \in \sR_1$, we have $\rho \circ \Phi_R(\omega, \omega', \v)=\Phi \circ \rho(\omega, \omega', \v).$
\item If $(\omega, \omega', \v) \in \sR_1$ has a stable periodic orbit of period larger than four, then $\rho(\omega, \omega', \v)$ has a stable periodic
orbit of strictly smaller period.
\end{enumerate}
\end{theorem}

Statement (3) has the following consequence for invariant measures.
\begin{corollary}[Renormalization acts on measures]
\name{cor:renormalization acts on measures}
Let $\nu$ be a $\Phi$-invariant Borel measure on $X$. Then $\nu \circ \rho^{-1}$ is also a
$\Phi$-invariant Borel measure. 
\end{corollary}
\begin{proof}
Let $A \subset X$ be Borel. Then statement (3) implies that
$\nu \circ \rho^{-1} \circ \Phi^{-1} (A)=\nu \circ \Phi_R^{-1} \circ \rho^{-1}(A).$
Since $\nu$ is $\Phi$ invariant, it's restriction to $\sR_1=\rho^{-1}(X)$ is also $\Phi_R$ invariant. So, $\nu \circ \rho^{-1} \circ \Phi^{-1} (A)=\nu \circ \rho^{-1}(A)$ as desired.
\end{proof}

Statements (2) and (4) of the Renormalization Theorem are useful for detecting stable periodic orbits. 
We will say that $x \in X$ is {\em $n$ times renormalizable}
if $\rho^k(x)$ is well defined for all $k=1, \ldots, n$. 
We use $\sR_n \subset X$ to denote the set of $(\omega, \omega', \v)$ which are $n$ times renormalizable.
So, $\sR_n=\rho^{-n}(X)$ for $n \geq 0$. 
Similarly, we say {\em the orbit of $x$ is  $n$ times renormalizable}
if there is an $m \in \Z$ such that $\Phi^m(x)$ is $n$ times renormalizable.
We use $\sO_n$ to denote the set of all  $x \in X$ whose orbit is $n$ times renormalizable.
Note that $\sO_n$ is the smallest $\Phi$-invariant set containing $\sR_n$. 

Suppose that $x \in \sO_n \sm \sO_{n+1}$.
By (2), the only explanation is that each possible $y=\rho^n \circ \Phi^m(x)$
lies in $\NUC$ or $P_4$. 
In the latter case, we conclude by (4) that $x$ has a stable periodic orbit. We would like to make this conclusion
hold almost always, so we make the following definition.

\begin{definition}
Let $\nu$ be a Borel measure on $X$. We say $\nu$ is {\em robustly renormalizable} if
for all integers $n \geq 0$ we have $\nu \circ \rho^{-n}(\NUC)=0$.
\end{definition}

Assuming $\nu=\mu \times \mu' \times \mu_N$ as in Proposition \ref{prop:product measures}, this is equivalent to the measures of the form $\mu \circ c^{-n}$ and $\mu' \circ c^{-n}$ never having an atom at $\omegaalt$. See Proposition \ref{prop:omegaalt2}.

\begin{corollary}
\name{cor:limit}
Suppose $\nu$ is robustly renormalizable. Then,
$\nu(\NS)=\lim_{n \to \infty} \nu(\sO_n).$
\end{corollary}
\begin{proof}
Since $\nu$ is robustly renormalizable, by statement (4) of the Renormalization Theorem, we know that for $\nu$-almost every
$x \in X \sm \NS$ there are $m,n \in \Z$ with $n \geq 0$ so that $\rho^n \circ \Phi^m(x) \in P_4$. 
And conversely, as stated above if $\rho^n \circ \Phi^m(x) \in P_4$ we know that $x \in X \sm NS$. 
Therefore $X \sm \NS$ is the smallest $\Phi$-invariant set containing $\bigcup_{n=0}^\infty \rho^{-n}(P_4)$.
Observe that $\rho^{-n}(P_4)=\sR_n \sm \sR_{n+1}$ up to a set of $\nu$-measure zero so that
$$X \sm \NS=\bigcup_{n=0}^\infty (\sO_n \sm \sO_{n+1}) \and \NS=\bigcap_{n=0}^\infty \sO_n.$$
This is a nested intersection, so the conclusion follows.
\end{proof}

Because of this Corollary, we wish to iteratively compute the measures of the sets  $\sO_n$. For this, we need some understanding of the return times to $\sR_n$. 
For non-negative $n \in \Z$, the {\em return time} of an element $x \in \sR_n$ to $\sR_n$ is the smallest integer $m>0$ for which 
$\Phi^m(x) \in \sR_n$. We write $m=\ret{n}{x}$. The existence of this number is provided by statement (2) of the Renormalization Theorem.
Observe that if $\nu$ is $\Phi$-invariant then we have 
$$\nu(\sO_n)=\int_{\sR_n} \ret{n}{x}~d\nu(x).$$
This demonstrates the importance of knowing the return times. 

The following lemma explains how return time is related to 
the insertion operation. Let $(\eta, \eta', \v)=\rho(\omega, \omega', \v)$ so that $\eta=c(\omega)$ and $\eta'=c(\omega')$. Recall from 
Proposition \ref{prop:inverses of collapsing map} that $\omega$ can be recovered from $\eta$ by an insertion operation.
That is, $\omega=\sI_f(\eta)$ for some $f:\Z \to \sW$ with $f(i)=(-+)^{n_i}$ for all $i$. Similarly, $\omega'=\sI_f(\eta')$ for some $f':\Z \to \sW$ with $f'(i)=(-+)^{n'_i}$ for all $i$. 

\begin{lemma}[Return times]
\name{lem:returns}
With $(\omega, \omega', \v) \in \sR_1$ as above, let $\v=(a,b)$ and let $\v'=(sb,sa)$, which is the directional
component of $\Phi(\omega, \omega', \v)$ as in equation \ref{eq:Phi2}. Assume $\omega=\sI_f \circ c(\omega)$
and $\omega'=\sI_{f'} \circ c(\omega')$ as above. The following statements give the return time of $(\omega, \omega', \v)$ to $\sR_1$.
\begin{itemize}
\item If $\v'=(1,0)$ then $\ret{1}{x}=2 \ell\big(f(0)\big)+1$. 
\item If $\v'=(-1,0)$ then $\ret{1}{x}=2 \ell\big(f(-1)\big)+1$. 
\item If $\v'=(0,1)$ then $\ret{1}{x}=2 \ell\big(f'(0)\big)+1$.
\item If $\v'=(0,-1)$ then $\ret{1}{x}=2 \ell\big(f'(-1)\big)+1$. 
\end{itemize}
\end{lemma}
Observe that since each $f(i)=(-+)^{n_i}$, these word lengths are all multiples of $2$ and the return times are all equivalent to one modulo four.

We now generalize the return time definition to a linear operator on the space $\sM$ of all Borel measurable functions on $X$. Suppose $f\in \sM$. We define the {\em retraction of $f$ to $\sR_n$} to be the function 
$r_n(f):\sR_n \to \R$ given by 
$$r_n(f; x)=\sum_{i=0}^{\ret{n}{x}-1} f \circ \Phi^i(x).$$
We think of this as a generalization of the return time, since for the constant function $\bbone$ we have $\ret{n}{x}=r_n(\bbone; x)$. Observe that by $\Phi$-invariance of $\nu$ we have
\begin{equation}
\name{eq:retraction identity}
\int_{\sO_n} f~d\nu=\int_{\sR_n} r_n(f; x)~\d\nu.
\end{equation}

The theory of conditional expectations demonstrates that for any positive $\nu$-integrable $f:\sR_n \to \R$, there is a Borel measurable function $g:X \to \R$ so that 
$$\int_{\sR_n} f(y)~d\nu(y)=\int_{X} g(x)~d\nu \circ \rho^{-n}(x).$$
This function $g$ is the {\em conditional expectation} of $f(x)$ given $\rho^n(x)$, and it is uniquely defined $\nu \circ \rho^{-n}$-a.e.. See Chapter 27 of \cite{Loeve78}, for example. 

We now combine the theory of conditional expectations with our retraction operator. For $f:X \to \R$ Borel measurable and $\nu$ a $\Phi$-invariant, define 
$C(\nu, n)(f):X \to \R$ to be the {\em conditional expectation of the retraction of $f$ to $\sR_n$}. That is, $C(\nu, n)(f)$ is the $\nu \circ \rho^{-n}$-a.e. unique Borel measurable function such that for all Borel $B \subset X$ we have
\begin{equation}
\name{eq:cond exp integral}
\int_{\rho^{-n}(B)} r_n(f;x)~d\nu(x)=\int_{B} C(\nu, n)(f)(x)~d\nu \circ \rho^{-n}(x).
\end{equation}
In particular, taking $B=X$ we see $C(\nu, n):L^1(\nu) \to L^1(\nu \circ \rho^{-n})$.
Observe that $C(\nu,0)$ may be interpreted is the identity operator $\nu$-a.e., and $C(\nu,n)$ satisfies the following cocycle equation
for integers $m,n \geq 0$.
\begin{equation}
\name{eq:cocycle composition}
C(\nu,m+n)(f)=C(\nu \circ \rho^{-n},m) \circ C(\nu,n)(f) \qquad \textrm{$\nu \circ \rho^{-m-n}$-a.e.}
\end{equation}
In particular, we may write 
\begin{equation}
\name{eq:expand cocycle}
C(\nu,n)=C(\nu \circ \rho^{n-1}, 1) \circ C(\nu \circ \rho^{n-2}, 1) \circ \cdots \circ C(\nu, 1).
\end{equation}
So to evaluate the cocycle is is sufficient understand the operators $C(\nu \circ \rho^{n}, 1)$.

We summarize our approach by restating Corollary \ref{cor:limit} with this language, which is our approach to proving 
results such as Theorem \ref{thm:random tilings}. (This will be equally important for similar results involving polygon exchange maps in \cite{H12}.) Observe that by Equations \ref{eq:retraction identity} and \ref{eq:cond exp integral}, for any Borel measurable $f$ we have 
\begin{equation}
\name{eq:main approach}
\int_{\sO_n} f~d\nu=\int_{\sR_n} r_n(f,x)~d\nu(x)=\int_X C(\nu,n)(f)(x)~d \nu \circ \rho^{-n}(x).
\end{equation}
Specializing to the constant function, we obtain the following.

\begin{corollary}[Main Approach]
\name{cor:approach}
Suppose $\nu$ is robustly renormalizable. Let ${\bbone}(x)=1$ for all $x \in X$. 
Then for all $n$,
$$\nu(\sO_n)=\int_X C(\nu,n)(\bbone)(x)~d \nu \circ \rho^{-n}(x).$$
and $\nu(\NS)=\lim_{n \to \infty} \nu(\sO_n).$
\end{corollary}

For general $\nu$, this limit is probably impossible to evaluate. But, for the measures associated to Theorem \ref{thm:random tilings},
the situation is very nice. In particular, we will be studying robustly renormalizable measures $\nu$ with the property that there is a finite dimensional
subspace $L \subset \sM$ such that ${\bbone} \in L$ and 
$C(\nu,n)(f) \in L$ for all $f \in L$. 
Such a reduction reduces the problem to studying a cocycle over a finite dimensional vector space.

In the cases of interest, we partition the space $X$ into six non-empty pieces $\Step_1, \Step_2, \ldots, \Step_6$ and define
the linear embedding $\epsilon:\R^6 \to \sM$ by $\epsilon(\bp)(x)=\bp_i$ if $x \in \Step_i$, and 
\begin{equation}
\name{eq:L}
L=\set{\epsilon(\bp)}{$\bp \in \R^6$}.
\end{equation}
We say $x \in \Step_i$ has {\em step class $i$}. This term is defined in the three paragraphs below.

Let $x=(\omega, \omega', \v) \in X$ be arbitrary. A {\em step} will indicate information which depends only on information at the origin, 
and the next square visited by the curve through the normal $\v$ leaving the square centered at the origin. 
Define $(\eta, \eta', \v')=\Phi(\omega, \omega',\v)$. If $\v'=(1,0)$ or $\v'=(-1,0)$, we call $x$ a {\em horizontal step},
otherwise we call $x$ a {\em vertical step}. Note that if $x$ is part of a horizontal step, then $\omega'=\eta'$, but $\eta=\sigma^{sb}(\omega)$ with $s=\omega_0 \omega_0'$ so that 
$sb \in \{\pm 1\}$. See equation \ref{eq:Phi2}. 

We divide the classes of horizontal and vertical steps into smaller classes.
We call
$x=(\omega, \omega', \v)$ a {\em $\omega_0 \omega_1$-horizontal step} if $sb=1$ and a {\em $\omega_{-1} \omega_0$-horizontal step} if $sb=-1$. 
So, we have defined the term {\em $w$-horizontal step} for each word $w$ of length $2$, i.e. $w \in \{--,-+,+-,++\}$. 
If $x$ is either a $++$-horizontal steps or a $--$-horizontal steps, then we call $x$ a {\em matching horizontal step}.
Similarly, if $x$ is a vertical step, then $\eta'=\sigma^{sa} (\omega')$.
If $sa=1$ we define $x$ to be a {\em $\omega_0' \omega_1'$-vertical step}, and if
$sa=-1$ we define $x$ to be a {\em $\omega_{-1}' \omega_0'$-vertical step}.
If $x$ is either a $++$-vertical step or a $--$-vertical step, we call $x$ a {\em matching vertical step}. 

We say $x$ has {\em step class $1$} if $x$ is a $-+$-horizontal step,
has {\em step class $2$} if $x$ is a $+-$-horizontal step,
has {\em step class $3$} if $x$ is a matching horizontal step,
has {\em step class $4$} if $x$ is a $-+$-vertical step,
has {\em step class $5$} if $x$ is a $+-$-vertical step, and
has {\em step class $6$} if $x$ is a matching vertical step. This defines the six sets $\Step_1, \ldots \Step_6 \subset X$, and
the subspace $L \subset \sM$ as in equation \ref{eq:L}.

For the following lemma define $\chi_i$ to be the characteristic function of $\Step_i$, and define $\be_i \in \R^6$ to be the 
standard basis vector with $1$ in position $i$. 

\begin{lemma}[Collapsed Steps]
\name{lem:returns2}
Suppose $x \in \sR_1$ has return time $\ret{1}{x}=4m+1$ and $\rho(x) \in \Step_j$. Then, for all $i \in \{1, \ldots, 6\}$ we have 
$r_1(\chi_i;x)=\be_j^T M \be_i$ with
$$M=\left[\begin{array}{rrrrrr}
m & m-1 & 2 & 0 & 0 & 2m \\
m & m+1 & 0 & 0 & 0 & 2m \\
m & m & 1 & 0 & 0 & 2m \\
0 & 0 & 2m & m & m-1 & 2 \\
0 & 0 & 2m & m & m+1 & 0 \\
0 & 0 & 2m & m & m & 1 \\
\end{array}\right].$$
\end{lemma}

Observe that the $j$-th row of $M$ give the number of each step type which appears in the set $\{x, \Phi(x), \ldots, \Phi^{4m+1}\}$
provided $\rho(x) \in \Step_j$ and $\ret{1}{x}=4m+1$. This lemma is restated and proved as Lemma \ref{lem:returns4}.

The following Corollary is a main tool in the proofs of Theorem \ref{thm:random tilings} and for subsequent work
on polygon exchange maps \cite{H12}.

\begin{corollary}
\name{cor:matrix formula}
Suppose $\nu$ is a measure on $X$ such that for all $j \in \{1, \ldots, 6\}$ the conditional expectation of $\ret{1}{x}$ given that $\rho(x) \in \Step_j$ is a constant $4m_j+1$ at
$\nu$-a.e. point of $\rho^{-1}(\Step_j)$. 
Then $C(\nu,1) \circ \epsilon(\bp) \in L$ for all $\bp \in L$, and $\nu \circ \rho^{-1}$-a.e. we have $C(\nu,1) \circ \epsilon(\bp)=\epsilon \big( M' \bp\big)$ where
$$M'=\left[\begin{array}{rrrrrr}
m_1 & m_1-1 & 2 & 0 & 0 & 2m_1 \\
m_2 & m_2+1 & 0 & 0 & 0 & 2m_2 \\
m_3 & m_3 & 1 & 0 & 0 & 2m_3 \\
0 & 0 & 2m_4 & m_4 & m_4-1 & 2 \\
0 & 0 & 2m_5 & m_5 & m_5+1 & 0 \\
0 & 0 & 2m_6 & m_6 & m_6 & 1 \\
\end{array}\right].$$
\end{corollary}
\begin{proof}
Since $r_1(\chi_i;x)=\be_j \cdot M \be_i$ for $x \in \Step_j$ with $M$ as in the above lemma and $m=m(x)$ defined so that $\ret{1}{x}=4m+1$.
Note then that $r_1(\chi_i;x)=a_{i,j} m(x)+b_{i,j}$ for some $a_{i,j},b_{i,j} \in \Z$ coming from the entries of $M$. 
Then, by definition of $C(\nu,1)$, for all $i, j \in \{1, \ldots, 6\}$ and all $B \subset \Step_j$ we have
\begin{equation}
\name{eq:proof1}
\int_{\rho^{-1}(B)} a_{i,j} m(x)+b_{i,j}~d\nu(x)=\int_{B} C(\nu, 1)(\chi_i)(x)~d\nu \circ \rho^{-n}(x).
\end{equation}
By definition of $m_j$ and because we defined $m(x)$ so that $\ret{1}{x}=4m(x)+1$, 
\begin{equation}
\int_{\rho^{-1}(B)} 4m(x)+1~d\nu(x)=\int_{B} 4m_j+1~d\nu \circ \rho^{-n}(x) \quad \textrm{for all Borel $B \subset \Step_j$}.
\end{equation}
Therefore, we see that $C(\nu, 1)(\chi_i)(x)=a_{i,j} m_j+b_{i,j}$ satisfies the equation \ref{eq:proof1}. This determines $C(\nu, 1)(\chi_i)(x)$ by a.e.-uniqueness of conditional expectations.
Finally, we can extend to all of the subspace $L \subset \sM$ by linearity of the operator $C(\nu,1)$. 
\end{proof}

Observe that to evaluate the integral of a function of the form $\epsilon(\bp)$ with respect to some measure $\nu$, it suffices to know the measures of the step classes. That is, 
\begin{equation}
\name{eq:dot product integral}
\int_X \epsilon(\bp)(x)~d\nu=\sum_{i=1}^6 \nu(\Step_i) \bp_i.
\end{equation}
We give formulas for these measures in two cases below.

\begin{proposition}
\name{prop:step characterization of P4}
Let $\nu=\mu \times \mu' \times \mu_N$ where $\mu$ and $\mu'$ are shift invariant measure on $\Omega_\pm$. Then,
$$\nu(\Step_1)=\frac{1}{2} \mu\big(\sC(-+)\big), \quad 
\nu(\Step_2)=\frac{1}{2} \mu\big(\sC(+-)\big), \and
\nu(\Step_3)=\frac{1}{2} \mu\big(\sC(--) \cup \sC(++)\big).$$
$$\nu(\Step_4)=\frac{1}{2} \mu'\big(\sC(-+)\big), \quad 
\nu(\Step_5)=\frac{1}{2} \mu'\big(\sC(+-)\big), \and
\nu(\Step_6)=\frac{1}{2} \mu'\big(\sC(--) \cup \sC(++)\big).$$
In addition, for the restriction of $\nu$ to $P_4$ we have 
$$\nu(\Step_1 \cap P_4)=\nu(\Step_5 \cap P_4)=\mu \big( \sC(-+)\big) \mu' \big( \sC(+-)\big), \and$$
$$\nu(\Step_2 \cap P_4)=\nu(\Step_4 \cap P_4)=\mu \big( \sC(+-)\big) \mu' \big( \sC(-+)\big).$$
Also $\nu(\Step_3 \cap P_4)=\nu(\Step_6 \cap P_4)=0$. 
\end{proposition}

\subsection{Proofs of the Renormalization Theorems}
\name{sect:proofs of renormalization}
In this section we prove the renormalization theorems of the previous section. Since this section contains no substantially new results, it may be skipped by a reader disinterested in
these proofs. If the reader just wishes to understand why the system is renormalizable, he/she may read the first few paragraphs of the section and attempt to understand Figures \ref{fig:renormalize} and \ref{fig:even rectangles}.

Our proofs are primarily based on analysis of tilings rather than the map $\Phi:X\to X$. This is motivated by the natural point of view that
the dynamics of $\Phi$ correspond to following curves in the associated tiling. All proofs in this section utilize this philosophy. Therefore, we fix a tiling $\tiling{\tau_{\omega, \omega'}}$ once and for all. 

In the definition of the collapsing map, we utilized a subset $K(\omega)\subset \Z$ which indicated
the places to remove symbols under the collapsing map. See equation \ref{eq:K}, and observe that
$K(\omega)$ is well defined for all $\omega \in \Omega_\pm$. We define
$\barK=K(\omega) \times K(\omega')$ and call this set the {\em(centers of the) kept squares}. 

The squares with centers in $\Z^2 \sm \barK$ form a family of rows and columns. These are the gray squares in Figures \ref{fig:renormalize} and \ref{fig:even rectangles}. On the level of tilings,
the tiling associated to the image $\rho(\omega, \omega', \v)$ is $\tiling{\tau_{c(\omega), c(\omega')}}$.
This tiling can be obtained from $\tiling{\tau_{\omega, \omega'}}$ by removing all columns consisting of squares whose centers lie in $\Z^2 \sm \barK$.
Then slide the remaining columns together to reconstruct a Truchet tiling. The remaining squares whose centers were in $\Z^2 \sm \barK$ consist of a family of rows of squares. Again, remove these and slide the rows together. 

\begin{remark}
\name{rem:collapsibility and proofs}
Observe that the renormalized tiling is defined modulo translation if $\omega$ and $\omega'$ are
unbounded-collapsible. To define the new tiling precisely, we need $\omega$ and $\omega'$ to also be zero-collapsible.
In this section we will only need $\omega$ and $\omega'$ to be
unbounded-collapsible.\end{remark}

\begin{figure}
\begin{center}
\includegraphics[width=3.5in]{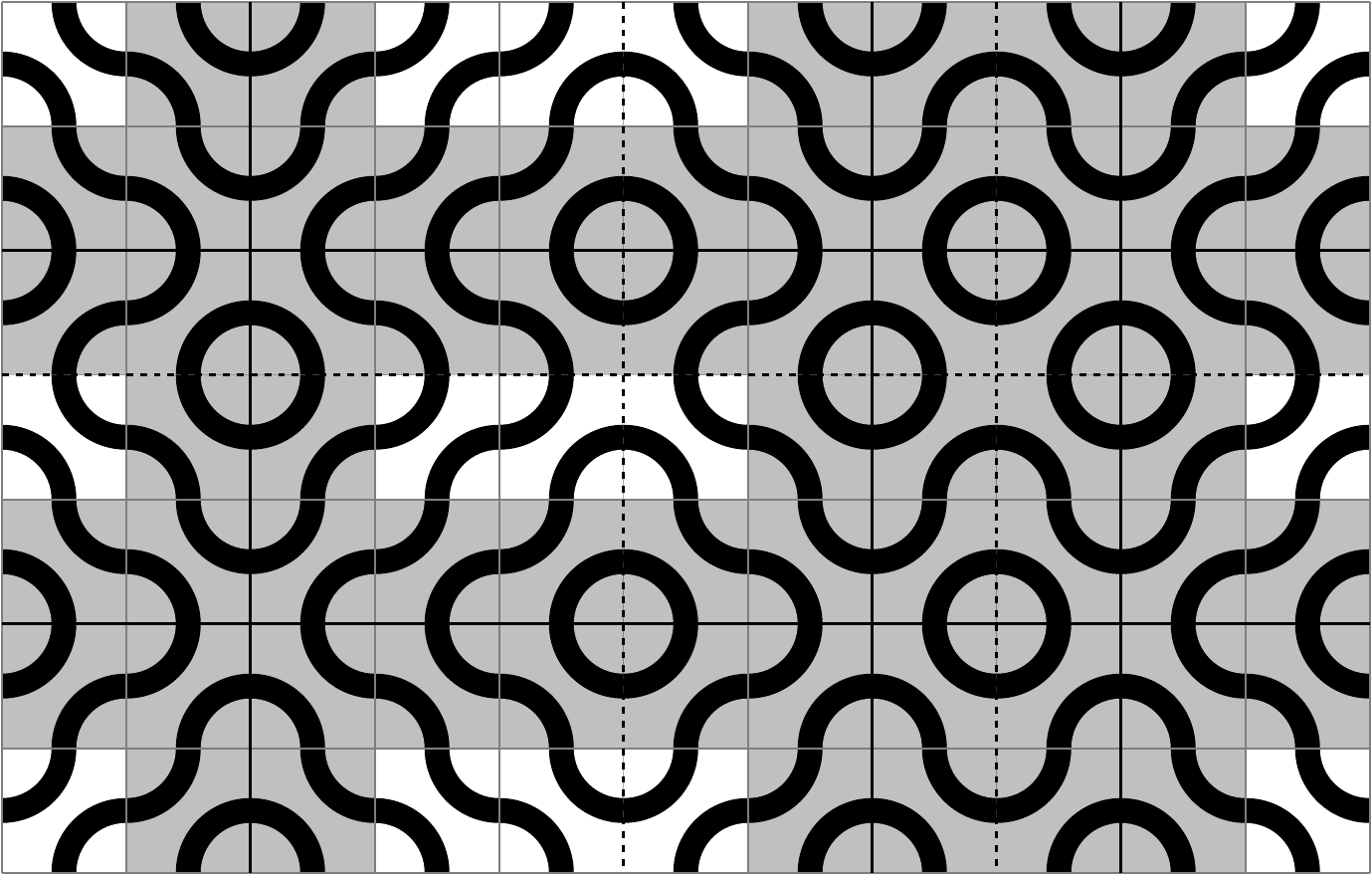}
\caption{The set $\barK$ consists of the indices of white squares.}
\name{fig:even rectangles}
\end{center}
\end{figure}

To analyze this process, we introduce the following definition.

\begin{definition}
Let $m$ be an integer. The line $x=m+\frac{1}{2}$ is a {\em vertical $-+$-dividing line} if
$\omega_{m}=-1$ and $\omega_{m+1}=1$. Similarly, $x=m+\frac{1}{2}$ is a {\em vertical $+-$-dividing line} if
$\omega_{m}=1$ and $\omega_{m+1}=-1$. The line $y=m+\frac{1}{2}$ is a {\em horizontal $-+$- or $+-$-dividing line} if
$\omega'_{m}=-1$ and $\omega'_{m+1}=1$ or $\omega'_{m}=1$ and $\omega'_{m+1}=-1$, respectively.
\end{definition}

The $-+$-dividing line for the tiling shown in figure \ref{fig:even rectangles} are drawn as darkened solid lines.
The $+-$-dividing lines are shown as dashed lines.
The following can be observed by the reader.

\begin{proposition}
\name{prop:barK}
The squares indexed by elements of $\Z^2 \sm \barK$ are those squares in the $1$-neighborhood of a $-+$-dividing line.
\end{proposition}

We can also connect these dividing lines to the paths which close up after visiting exactly four squares.

\begin{proposition}
\name{prop:loop4}
An arc of a tile in the tiling $\tiling{\tau_{\omega, \omega'}}$ is part of a closed path which visits only four squares if and only if the arc joins two dividing lines with differing signs (e.g. the arc starts on a $-+$-horizontal dividing line and ends on $+-$-vertical dividing line.)
\end{proposition}

This statement also implies statement (1) of the \hyperref[thm:renormalization]{Renormalization Theorem}.

\begin{proof}
Consider such a path in $\tiling{\tau_{\omega, \omega'}}$. Observe that such a path corresponds to 
a choice of $m,n \in \Z$ for which $\omega_m \omega'_n=-1=\omega_{m+1} \omega'_{n+1}$
and $\omega_{m+1} \omega'_n=1=\omega_m \omega'_{n+1}$. Therefore either
$\omega_m=1=\omega'_{n+1}$ and $\omega_{m+1}=-1=\omega'_{n}$, or $\omega_{m+1}=1=\omega'_{n}$ and 
$\omega_m=-1=\omega'_{n+1}$. See the figure below, where values of $\omega$ are written below the tiles,
and values of $\omega'$ are written to the left.
\begin{center}
\includegraphics[height=1in]{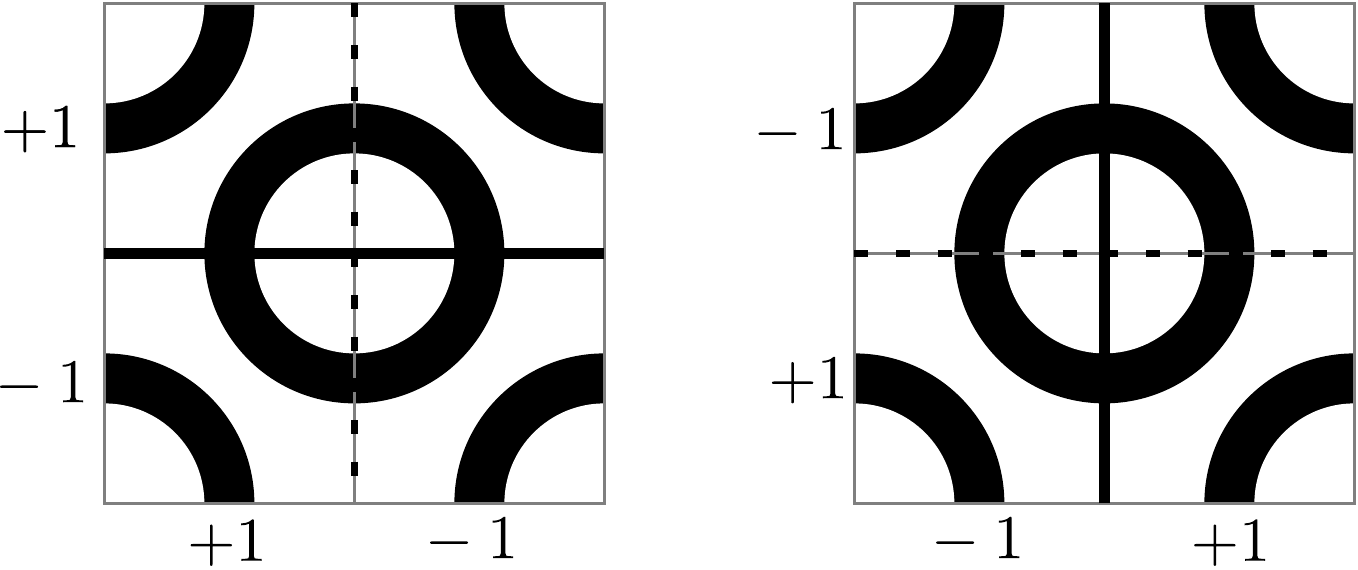}
\end{center}
In both cases, two dividing lines of opposite sign description pass through the center of the circular path.
\end{proof}

We can also use this observation to prove Proposition \ref{prop:step characterization of P4}.
\begin{proof}[Proof of Proposition \ref{prop:step characterization of P4}]
The values of $\nu(\Step_i)$ follow trivially from the definitions of the step classes. We concentrate on 
the intersections of the step classes with $P_4$. 
Suppose $x \in \Step_1 \cap P_4$, i.e. $x$ is a $-+$-horizontal step. Then, there are four possible local pictures
for the loop associated to $x$. These possibilities are shown below.
\begin{center}
\includegraphics[height=1in]{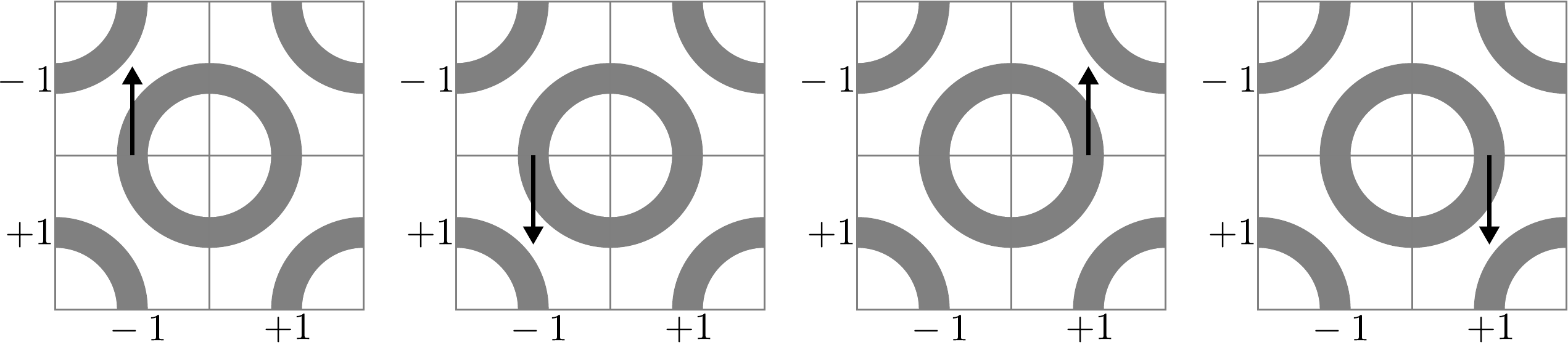}
\end{center}
Above, the arrow lies in the square which should be centered at the origin and represents the choice of the normal vector. Observe that the probability of any of these local pictures occurring is given by $\frac{1}{4} \mu \big( \sC(-+)\big) \mu' \big( \sC(+-)\big)$, with the factor of $\frac{1}{4}$ coming from $\mu_N$. The other cases are similar. 
\end{proof}

The following lemma is the main observation in the proofs of the remaining Renormalization results. 

\begin{lemma}[Crossing dividing lines]
\name{lem:crossing dividing lines}
Let $\alpha$ be an arc of a square whose center lies in $\Z^2 \sm \barK$. If $\alpha$ is not part of a loop
which visits only four squares, then $\alpha$ is part of a path $\gamma$ contained in a $2 \times 2$ square which visits four squares and
joins the two boundaries components of either the $1$-neighborhood of a $-+$-horizontal or the $1$-neighborhood of a $-+$-vertical dividing line. 
In the first case, the endpoints of $\gamma$ differ by a vertical translation by $2$, the two endpoints do not lie in the  $1$-neighborhood of a $-+$-vertical dividing line,
and the values of $\omega$ taken on the $x$-coordinates of centers of the four squares visited by $\gamma$ agree.
In the second case, the endpoints of $\gamma$ differ by a horizontal translation by $2$, the two endpoints do not lie in the  $1$-neighborhood of a $-+$-horizontal dividing line,
and the values of $\omega'$ taken on the $y$-coordinates of centers of the four squares visited by $\gamma$ agree.
\end{lemma}

The lemma was crafted to indicate that the local picture of the tiling at the squares visited by $\gamma$ is determined by one of the four pictures below.
The possible curves $\gamma$ are drawn in black below. The left two cases cross $-+$-vertical dividing lines.
\begin{center}
\includegraphics[height=1in]{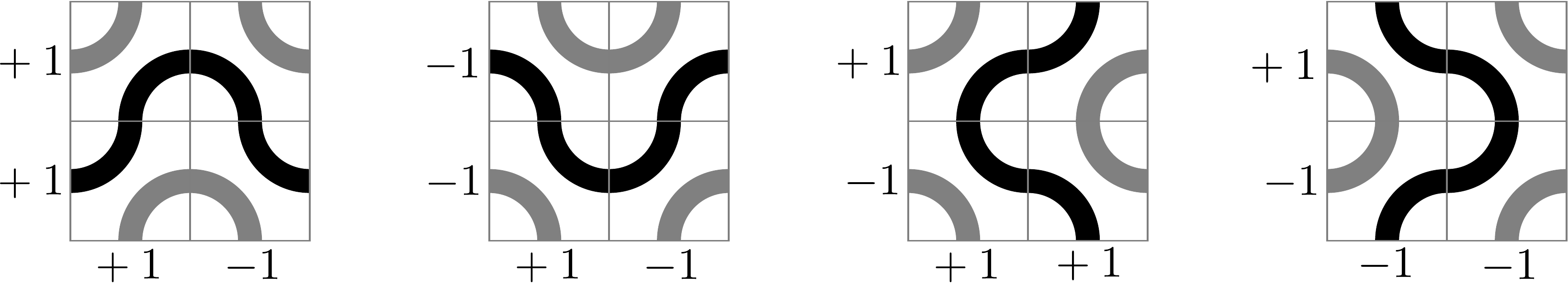}
\end{center}

\begin{proof}
Since $\alpha$ lies a square whose center lies in $\Z^2 \sm \barK$, it lies in the $1$-neighborhood 
of either a $-+$-horizontal dividing line or a $-+$-vertical dividing line. See Proposition \ref{prop:barK}.
We break into cases based on this. 

First suppose that $\alpha$ belongs to a $1$-neighborhood of a $-+$-vertical dividing line but not to a $1$-neighborhood of a $-+$-horizontal dividing line. Assume the center of the square containing $\alpha$ is $(m,n)$, and that this square lies to the left of this vertical line. Then, $\omega_m=-1$ and $\omega_{m+1}=1$. 
Consider the case when $\omega'_n=-1$. Since the square centered at $(m,n)$ is not in the $1$-neighborhood of 
a $-+$-horizontal dividing line, we have $\omega'_{n+1}=-1$. We draw two local pictures, with the left case
$\omega'_{n-1}=1$ and the right case $\omega'_{n-1}=1$. These cases are drawn on the left and right below, respectively, with a dot placed at $(m,n)$. 
\begin{center}
\includegraphics[height=1.25in]{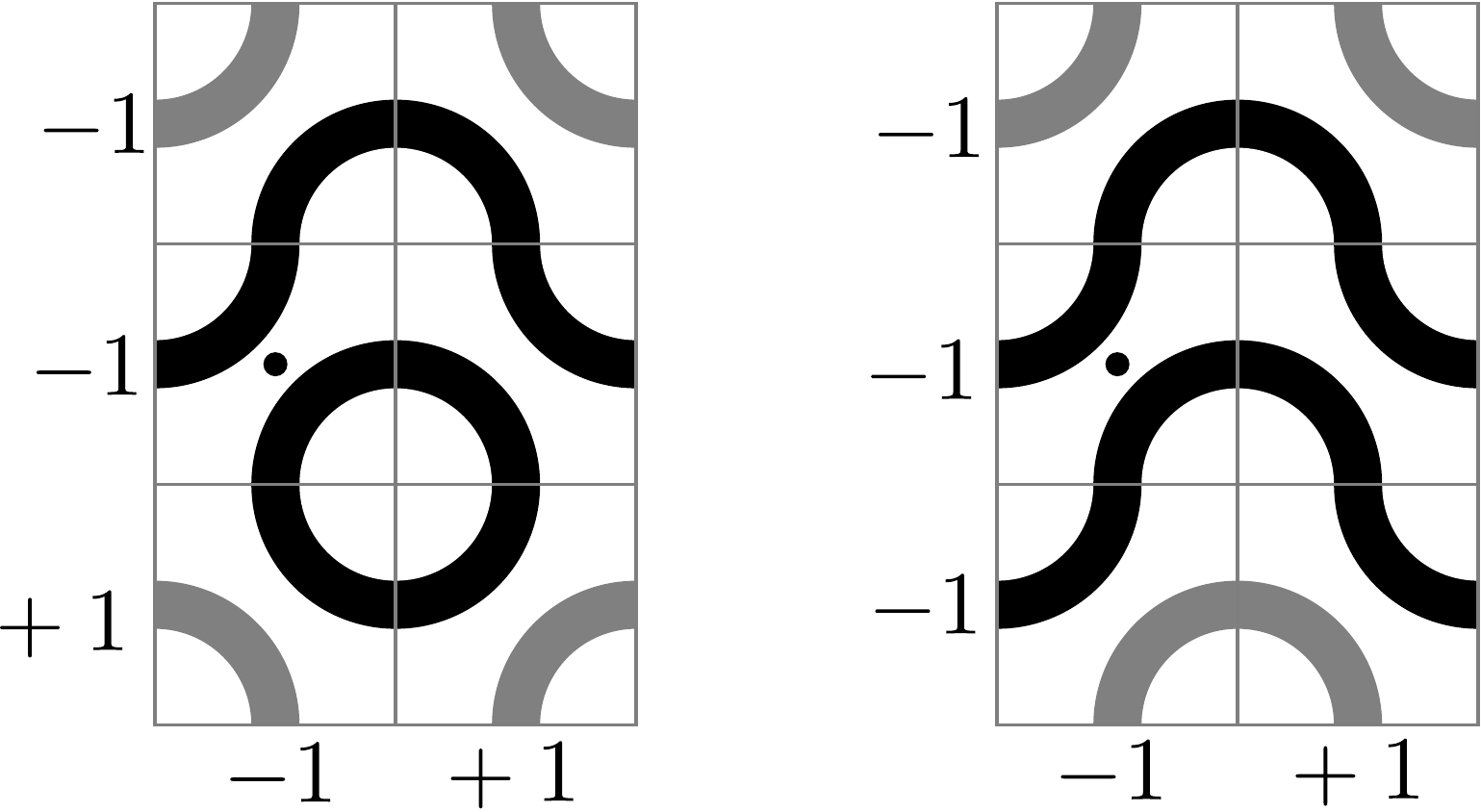}
\end{center}
The arcs $\alpha$ under consideration are the ones that start in the squares with a dot. Observe that unless
the arc is part of a loop of length four, the four black arcs connected to $\alpha$ join opposite sides of the 
$1$-neighborhood of the $-+$-vertical dividing line (the central vertical line). 
The case of $\omega'_n=1$ is similar, and produces the same pictures but reflected with respect to the $x$-axis.
The case where $(m,n)$ is to the right of the dividing line is also similar: just move the dot one square to the right.

The case where $\alpha$ belongs to the neighborhood of a $-+$-horizontal dividing line and
not to the neighborhood of a $-+$-vertical dividing line differs by a reflection in the line $y=x$. 

Finally, we suppose that the square with center $(m,n)$ containing $\alpha$ lies in both the $1$-neighborhood of a $-+$-horizontal dividing line and the $1$-neighborhood of a $-+$-vertical dividing line. The intersections of the $1$-neighborhoods form a $2 \times 2$ square as depicted on the left below.
\begin{center}
\includegraphics[height=1.25in]{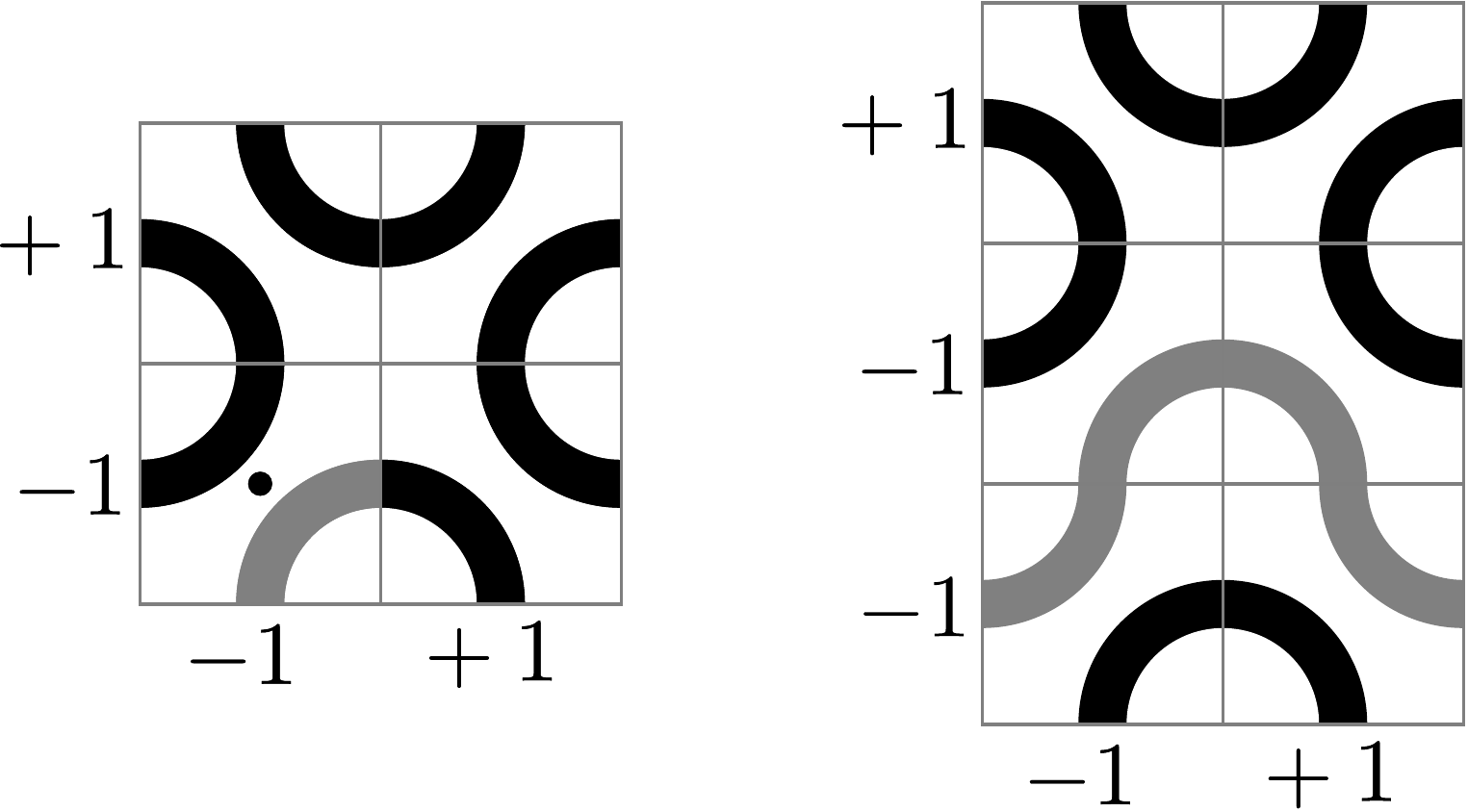}
\end{center}
We will assume that $\alpha$ is the gray arc in the above left picture, and $(m,n)$ is the coordinates of the dot. The other cases will be similar. 
In this case,  $\omega_m=-1$, $\omega_{m+1}=1$ and $\omega'_n=-1$. 
If $\omega'_{n-1}=1$, then $y=m-\frac{1}{2}$ is a $+-$-horizontal dividing line, and by Proposition \ref{prop:loop4} we know $\alpha$ is part of loop visiting four squares.
Therefore, we may assume $\omega'_{n-1}=-1$. In this case, the gray arc extends to a curve of length four crossing the neighborhood of the vertical dividing line as on the right side. 
\end{proof}

The following is equivalent to the remaining three statements of the \hyperref[thm:renormalization]{Renormalization Theorem}. 
\begin{theorem}
Assume $\omega$ and $\omega'$ are unbounded-collapsible.
\begin{enumerate}
\item[(2)] There is no infinite connected union of arcs of tiles contained entirely 
in squares with centers in $\Z^2 \sm \barK$. 
\item[(3)] If $\gamma$ is a curve of the associated tiling, then 
collapsing the $1$-neighborhoods of the $-+$-dividing lines slides the arcs of tiles with centers in $\barK$ together 
so that the images of the arcs of $\gamma$ form a curve in $\tiling{\tau_{c(\omega), c(\omega')}}$, with the natural cyclic or linear ordering
of these arcs respected. (The type of ordering depends if $\gamma$ is closed or bi-infinite.)
\item[(4)] No closed curve of the tiling is consists entirely of arcs of tiles with centers in $\barK$.
\end{enumerate} 
\end{theorem}
\begin{proof}
Consider statement (2). Suppose we had such a union of arcs. Choose one of the arcs, $\alpha$. By the Lemma, it is part of a union of four arcs $\gamma$
which crosses a $-+$-dividing line. Without loss of generality, assume that this dividing line is vertical. The endpoints of $\gamma$ differ by a horizontal translation by $2$
and lie outside of any $1-$-neighborhood of a $-+$-horizontal dividing lines. Without loss of generality, assume that the infinite union of arcs continues to the right. We observe that because
the endpoints were not in a $1-$-neighborhood of a $-+$-horizontal dividing lines, the next square to the right of the right endpoint of $\gamma$ cannot lie within the $1-$-neighborhood.
By assumption, this square must lie in the  $1-$-neighborhood of a $-+$-vertical dividing line. The curve traverses this neighborhood after visiting another four squares. We see by induction
that the infinite union of arcs crosses an contiguous infinite sequence of $1-$-neighborhoods of $-+$-vertical dividing lines. Therefore $\omega$ terminates in the infinite word $(-+)^\infty$,
which shows that $\omega$ is not unbounded-collapsible.

Statement (3) also follows from the lemma. The collections of four arcs crossing $1$-neighborhoods of $-+$-dividing lines
are collapsed.

Statement (4) follows from understanding the maximal contiguous collections of tiles with centers in $\barK$.
Such a collection arises from choosing integers $a \leq b$ and $c\leq d$ such that 
$\omega_a \ldots \omega_b$ and $\omega'_c \ldots \omega'_d$ consist of words in the set $\set{(+)^i(-)^j}{$i\geq 0$ and $j \geq 0$}$.
We leave it to the reader to check no such region contains closed loops. See the example below.
\begin{center}
\includegraphics[height=1.5in]{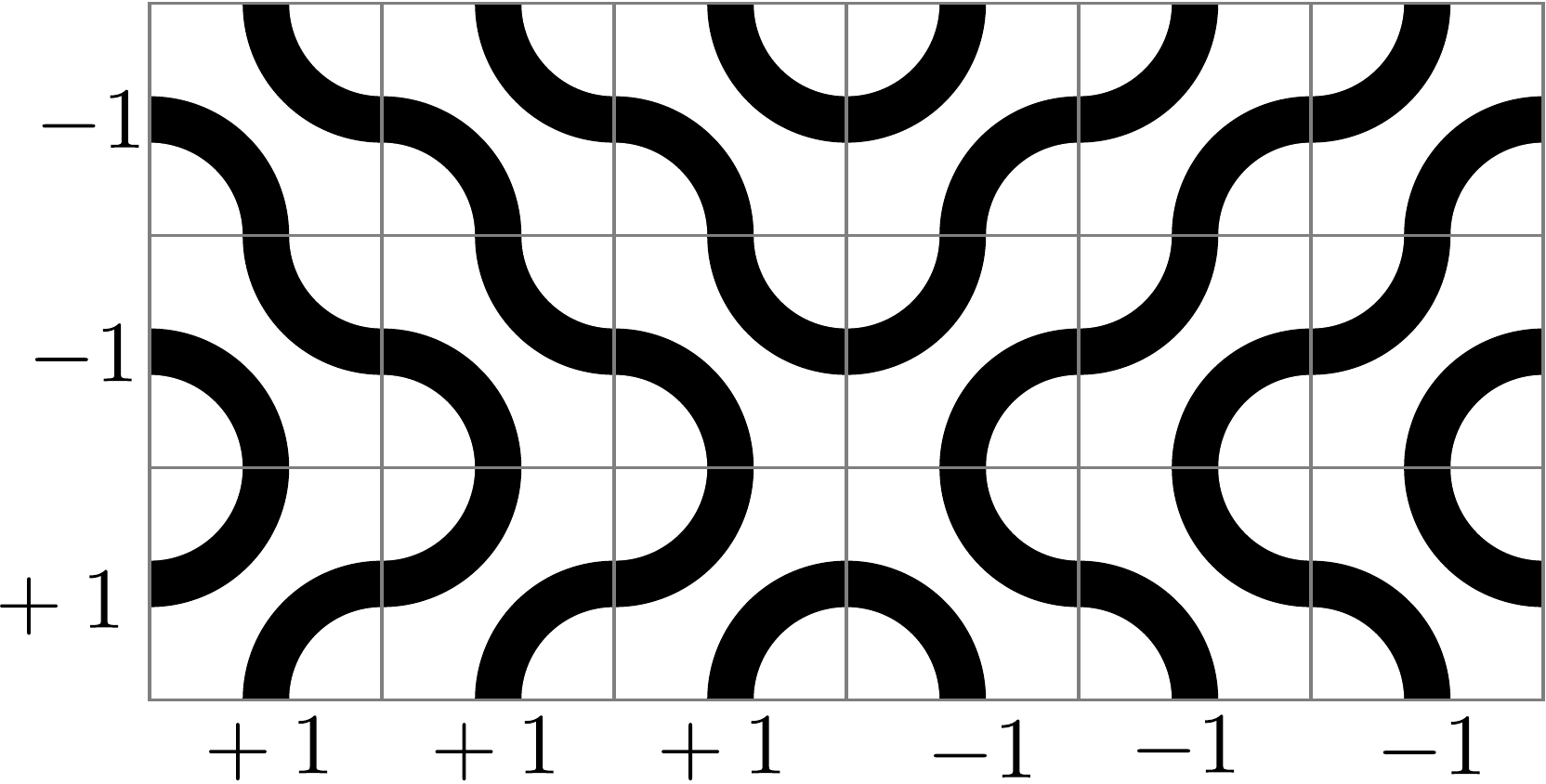}
\end{center}
\end{proof}

The following implies the \hyperref[lem:returns]{Return times lemma}.

\begin{lemma}[Restated Return times Lemma]
\name{lem:returns3}
Suppose that $\omega$ and $\omega'$ are unbounded-collapsible.
Suppose $(m,n) \in \barK$ and $\v'\in N=\{(1,0), (-1,0), (0,1), (0,-1)\}$ is a vector.
By the unbounded-collapsible property, we can define $k \in \Z$ as follows:
\begin{enumerate}
\item If $\v'=(1,0)$, $k$ is the smallest non-negative integer for which 
$\omega_{n+1} \ldots \omega_{n+2k}=(-+)^k$.
\item If $\v'=(-1,0)$, $k$ is the smallest non-negative integer for which 
$\omega_{n-2k} \ldots \omega_{n-1}=(-+)^k$.
\item If $\v'=(0,1)$, $k$ is the smallest non-negative integer for which 
$\omega'_{n+1} \ldots \omega'_{n+2k}=(-+)^k$.
\item If $\v'=(-1,0)$, $k$ is the smallest non-negative integer for which 
$\omega'_{n-2k} \ldots \omega'_{n-1}=(-+)^k$.
\end{enumerate}
Then, the curve of the tiling leaving the tile centered at $(m,n)$ and entering the 
tile centered at $(m,n)+\v'$ visits $4k$ tiles in $\Z^2 \sm \barK$ (starting at $(m,n)+\v'$) before returning to 
$\barK$. 
\end{lemma}
\begin{proof}
Assume we are in case (1), $\v'=(1,0)$. If $(m+1,n) \in \barK$, then by definition of $\barK$ we have 
$\omega_{m+1} \omega_{m+2} \neq -+$. So, $k=0$, and this agrees to the number of squares encountered before
returning to $\barK$: zero. Now suppose $(m+1,n) \not \in \barK$. Since $(m,n) \in \barK$, it must be that
$\omega_{m+1} \omega_{m+2} = -+$, so $k$ is at least one. We apply the Crossing Dividing Lines Lemma to
see that after visiting four squares in $\Z^2 \sm \barK$, the curve enters the tile centered at 
$(m+3,n)$. If this square is in $\barK$ we must have $\omega_{m+1} \ldots \omega_{m+4} \neq (-+)^2$,
and we have proved the Lemma with $k=1$. Otherwise, if $(m+3,n) \not \in \barK$ we have $\omega_{m+1} \ldots \omega_{m+4} = (-+)^2$. We may repeat this argument inductively. Since $\omega$ is unbounded-collapsible, eventually 
$\omega_{n+1} \ldots \omega_{n+2k} \neq (-+)^k$.
\end{proof}

Finally, we consider the \hyperref[lem:returns2]{Collapsed Steps Lemma}. For this we make a definition to parallel the definitions of ``steps.'' Indeed the following definition seems more natural. A {\em domino} is an unordered pair
$D=\{(m, n), (m', n')\}$ with $m, m', n, n' \in \Z$ and $(m,n)-(m',n') \in N$ (i.e. the squares
centered at these points are adjacent). We say $D$ is {\em horizontal} if 
$(m,n)-(m',n') \in \{(1,0), (-1,0)\}$ and {\em vertical} otherwise. 

A horizontal (resp. vertical) domino $D=\{(m, n), (m', n')\}$ is {\em in standard position} if $m'=m+1$ (resp. $n'=n$). For this paragraph we assume all dominoes are in standard position. 
Given $\omega$ and $\omega'$ as in this section, we say that 
a horizontal (resp. vertical) domino $D$ is a {\em $\omega_{m} \omega_{m+1}$-horizontal domino}
(resp. {\em $\omega'_{n} \omega'_{n+1}$-vertical domino}).
As with the definition of steps, the quantities $\omega_{m} \omega_{m+1}$ and $\omega'_{n} \omega'_{n+1}$ are viewed as words of length $2$. If $D$ is a $++$- or $--$-horizontal domino, we call it a {\em matching horizontal domino}. We make the analogous definition of {\em matching vertical domino}.

Suppose that $\gamma$ is a finite connected union of arcs of tiles. 
Then $\gamma$ visits a sequence of tiles with centers listed 
$(m_1, n_1), \ldots, (m_k, n_k)$, written in the order the tiles are visited.
Then each pair $\{(m_i,n_i), (m_{i+1}, n_{i+1})\}$ is a domino. We write $\sD(\gamma)$ for the collection
of such dominoes. (There are $k-1$ such dominoes.)
The {\em domino vector of the path $\gamma$}
is the element $\bd=(\bd_1, \ldots, \bd_6) \in \Z^6$, where 
$\bd_1$ is the number of $-+$-horizontal dominoes,
$\bd_2$ is the number of $+-$-horizontal dominoes,
$\bd_3$ is the number of matching horizontal dominoes,
$\bd_4$ is the number of $-+$-vertical dominoes,
$\bd_5$ is the number of $+-$-vertical dominoes, and
$\bd_6$ is the number of matching vertical dominoes.

We will now state a lemma equivalent to the \hyperref[lem:returns2]{Collapsed Steps Lemma}.
Suppose $\gamma$ is a finite connected union of at least two arcs of tiles
which begins and ends at a tile whose center lies in $\barK$, but for which no other 
visited tile has a center lying in $\barK$. By Lemma \ref{lem:returns3},
$\gamma$ visits a total of $4k+2$ squares for some integer $k \geq 0$. 
The renormalization removes all the squares of $\gamma$ except for the first and the last tile, and slides
these tiles together forming a domino. We call this the {\em domino of $\gamma$} and denote it $D_\gamma$. (This is well defined. See Remark \ref{rem:collapsibility and proofs}.)

\begin{lemma}[Collapsing Dominos Lemma]
\name{lem:returns4}
Suppose that $\omega$ and $\omega'$ are unbounded-collapsible.
Let $\gamma$ and $k$ be as in the previous paragraph. Then the domino vector $\bd$ of
$\gamma$ is determined by $k$ and the domino of $\gamma$, $D_\gamma$. 
\begin{enumerate}
\item If $D_\gamma$ is a $-+$-horizontal domino then $\bd=(k,k-1,2,0,0,2k)$.
\item If $D_\gamma$ is a $+-$-horizontal domino then $\bd=(k,k+1,0,0,0,2k)$.
\item If $D_\gamma$ is a matching horizontal domino then $\bd=(k,k,1,0,0,2k)$.
\item If $D_\gamma$ is a $-+$-vertical domino then $\bd=(0,0,2k, k,k-1,2)$.
\item If $D_\gamma$ is a $+-$-vertical domino then $\bd=(0,0,2k, k,k+1,2)$.
\item If $D_\gamma$ is a matching vertical domino then $\bd=(0,0,2k, k,k,1)$.
\end{enumerate}
\end{lemma}
\begin{proof}
This is again implied by Lemma \ref{lem:crossing dividing lines}. We consider the horizontal case. The vertical case follows by symmetry.  
In the case where $k=0$, we have that the collection of dominoes along the curve is $\sD(\gamma)=\{D_\gamma\}$. 
When $D_\gamma$ is a $-+$-horizontal domino it can not satisfy the conditions of the lemma because the whole domino lies in $\Z^2 \sm \barK$. 
This proves the case $k=0$.

Now suppose $k>0$. Orient $\gamma$ from left to right. We assume $\gamma$ starts at a tile centered at $(m,n)$. By Lemma \ref{lem:crossing dividing lines} it ends at a tile centered
at $(m+2k+1,n)$. This makes $D_\gamma$ a $\omega_m \omega_{m+2k+1}$-horizontal domino. We observe that the first domino in $\sD(\gamma)$
is a $\omega_m \omega_{m+1}=\omega_m -$-horizontal domino. Similarly, the last domino is a $+\omega_{m+2k+1}$-horizontal domino. 
The remaining collection of dominoes consists of $k$ $-+$-horizontal dominoes, $k-1$ $+-$-horizontal dominoes, and $2k$-matching vertical dominoes.
See the image below for the case when $\omega'$ takes the value $+1$ on path $\gamma$. By inspection, we observe that the first three possible conclusions 
of the lemma hold. 
\begin{center}
\includegraphics[height=1in]{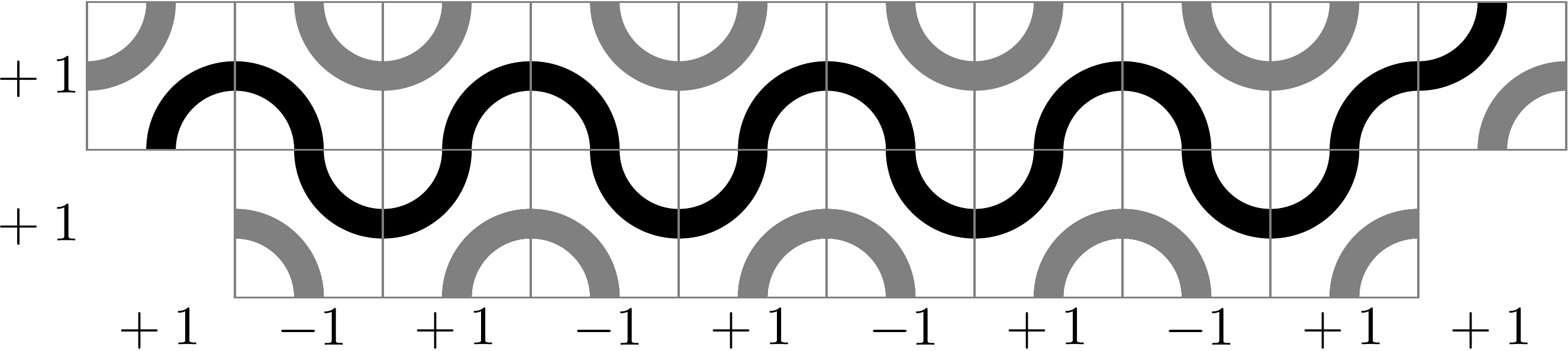}
\end{center}
\end{proof}

\section{Random tilings}
\name{sect:random}
In this section we prove theorem \ref{thm:random tilings} involving tilings generated by random elements of the shift space $\Omega_\pm$. 

For $0 \leq p \leq 1$, we will define a Borel probability measure $\mu_p$ on $\Omega_\pm$. The measure $\mu_p$ is the shift-invariant measure such that 
\begin{enumerate}
\item $\mu_p(\set{\omega \in \Omega_\pm}{$\omega_0=1$})=\frac{1}{2}$, 
\item the conditional expectation that $\omega_0=\omega_1$ given the values of $\omega_n$ for $n \leq 0$ is $p$, and
\item the conditional expectation that $\omega_0=\omega_1$ given the values of $\omega_n$ for $n \geq 1$ is $p$. 
\end{enumerate}
Statement (2) is equivalent to the statement that for any Borel set $A \subset \Omega_\pm$,
with the property that for any $\omega \in A$ and any function $f:\{1, 2, \ldots\} \to \{\pm 1\}$ the element $\eta^f \in A$ where
$\eta^f_n=\omega_n$ if $n \leq 0$ and $\eta^f_n=f(n)$ whenever $n>0$, we have 
\begin{equation}
\name{eq:conditional expectation}
\mu_p(\{\omega \in A~:~\omega_0=\omega_1\})=p \mu_p(A).
\end{equation}

This definition indicates that a random element of $\Omega_\pm$ taken with respect to $\mu_p$ can be constructed as described above Theorem \ref{thm:random tilings} in the introduction. Recall the definition of the cylinder set $\sC(f)$ in Section \ref{sect:topology}. We have the following. 

\begin{proposition}
\name{prop:measure uniquely defined}
Given any $p$ with $0 \leq p \leq 1$, there is a unique shift-invariant Borel probability measure $\mu_p$ satisfying statements (1)-(3) above. Let $f:\{m,\ldots, n\} \to \Z^2 \to \{\pm 1\}$ be arbitrary, and $\sC(f)$ be the associated cylinder set. Then,
\begin{equation}
\name{eq:measures of cylinder sets}
\mu_p\big(\sC(f)\big)=\frac{1}{2}p^k (1-p)^{m-n-k}
\end{equation}
where $k$ is the number of integers $i$ with $m \leq i <n$ and $f(i)=f(i+1)$. 
\end{proposition}
Note that by Carath\'eodory's extension theorem, the measure $\mu_p$ is uniquely determined by the measures of cylinder sets. We leave it to the reader to check
that $\mu_p\big(\sC(f)\big)$ must be as described, and that this function satisfies the conditions of  Carath\'eodory's theorem.

\subsection{The collapsing map}
Recall the definition of the collapsing map given in Section \ref{sect:collapsing}. In particular, recall Proposition \ref{prop:collapsing measures}: $\mu \circ c^{-1}$ is a shift-invariant measure whenever $\mu$ is. In fact the collection of measures $\set{\mu_p}{$0 < p \leq 1$}$ is invariant under this operation up to scaling. 

\begin{theorem}
\name{thm:img}
For $0 \leq p \leq 1$, let $q=\frac{1}{2-p}$. Then, $\mu_p \circ c^{-1}=p \mu_q.$
\end{theorem}

Observe that for $p \neq 0$, the measures $\mu_p \circ c^{-n}$ converge as $n \to \infty$ after being rescaled to to probability measures to the measure $\mu_1$. In particular, the renormalization dynamics here are not recurrent.

\begin{proof}
We ignore the cases where $p=0$ and $p=1$. The measure $\mu_0$ is supported on $\{\omegaalt, \sigma(\omegaalt)\}$, and
the measure $\mu_1$ is supported on the two constant sequences. These cases are trivial.

First we observe that $\mu_p \circ c^{-1}(\Omega_\pm)=p$. This will prove that $\mu_p \circ c^{-1}$ is $p$ times a probability measure. 
Equation  \ref{eq:measures of cylinder sets} can be used to demonstrate that the measure of the collection of those $\omega$ which are not unbounded collapsible is zero. Therefore, the measure of the collapsible elements of $\Omega_\pm$ equals the measure of the zero-collapsible elements. 
Recall that $C$ denotes the set of all collapsible elements of $\Omega_\pm$. 
By definition of zero-collapsible, 
the measure of the collection of non-zero-collapsible elements of $\Omega_\pm$ is the sum of the measures of two cylinder sets, so that
$$\mu_p(\Omega_\pm \sm C)=\mu\big(\sC(\widehat{-}+)\big)+\mu\big(\sC(-\widehat{+})\big)=1-p,$$
and $\mu_p \circ c^{-1}(\Omega_\pm)=\mu_p(C)=p$. 

Let $\nu$ be the probability measure $\frac{1}{p} \mu_p \circ c^{-1}$. We will prove that $\nu=\mu_q$. By Proposition \ref{prop:measure uniquely defined} it is sufficient to show that
$\nu$ satisfies statements (1)-(3) of the definition of $\mu_q$. Observe
$$\frac{1}{p} \mu_p \circ c^{-1}\big(\sC(\widehat{+})\big)=\frac{1}{p} \mu_p \big(\sC(+\widehat{+})\big)$$
by Corollary \ref{cor:preimages of cylinders}. Using Proposition \ref{prop:measure uniquely defined}, we evaluate
$\mu_p \big(\sC(+\widehat{+})\big)=\frac{p}{2}$ as desired.

We now prove that $\nu$ satisfies statement (2) in the definition of $\mu_q$. Statement (3) has a similar proof.
Define $\sA$ to be the $\sigma$-algebra generated by cylinder sets $\sC(f)$ where $f:\{m, \ldots, n\} \to \{\pm 1\}$ and $n \leq 0$.
By equation \ref{eq:conditional expectation}, it is enough to show that for all $B \in \sA$,
\begin{equation}
\name{eq:conditional expectation2}
\nu(\{\omega \in B~:~\omega_0=\omega_1\})=\frac{1}{2-p} \nu(B).
\end{equation}
Note that the right hand side of the equation then evaluates to $\frac{1}{p(2-p)} \mu_p \circ c^{-1}(B)$. We will show this equals the left hand side. 
Simplifying the left side yields
$$\nu(\{\omega \in B~:~\omega_0=\omega_1\})=\frac{1}{p} \mu_p \circ c^{-1} (\{\omega \in B~:~\omega_0=\omega_1\}).$$
We will try to understand $c^{-1} (\{\omega \in B~:~\omega_0=\omega_1\})$.
For each $n \geq 0$ define the two cylinder sets
$$S_n^+=\sC(+\widehat{+} (-+)^n +) \and S_n^-=\sC(\widehat{-} (-+)^n --).$$
From Corollary \ref{cor:preimages of cylinders}, and the definition of the collapsing map, we have
$$\bigcup_{n=1}^\infty S_n^+ \cup S_n^-=\set{\omega}{$\omega$ is collapsible and $c(\omega)_0=c(\omega)_1$} \quad \textrm{$\mu_p$-almost everywhere.}$$
(An $\omega \in S_n^\pm$ is zero-collapsible but not necessarily unbounded-collapsible,
and the non-unbounded-collapsible $\omega$ form a set of measure zero.) Since this union is disjoint, we have
\begin{equation}
\name{eq:summation}
\frac{1}{p} \mu_p \circ c^{-1} (\{\omega \in B~:~\omega_0=\omega_1\})=\frac{1}{p} \sum_{n=0}^\infty \mu_p\big(c^{-1}(B) \cap (S_n^+ \cup S_n^-) \big).
\end{equation}
We break $B$ into two disjoint pieces: $B_+=\set{\omega \in B}{$\omega_0=1$}$ and 
$B_-=\set{\omega \in B}{$\omega_0=-1$}$. 
Observe that $c^{-1}(B_+) \in \sA$ up to a set of $\mu_p$ measure zero (due to the necessity that
$c^{-1}(\omega)$ be unbounded collapsible on the right.) Since any $\omega \in c^{-1}(B_+)$ satisfies
$\omega_{-1}=\omega_0=1$, we have
$$c^{-1}(B_+) \cap S_n^+=c^{-1}(B_+) \cap \Big(\bigcap_{i=0}^{2n-1} \set{\omega}{$\omega_i \neq \omega_{i+1}$}\Big) \cap \set{\omega}{$\omega_{2n} = \omega_{2n+1}$}.$$
By inductively applying statement (2) in the definition of $\mu_p$ (or equation \ref{eq:conditional expectation}) and shift invariance of $\mu_p$ we have 
$\mu_p \circ c^{-1}(B_+ \cap S_n^+)=(1-p)^{2n} p \mu_p \circ c^{-1}(B_+).$ Note also that $B_+ \cap S_n^-=\emptyset$ for all $n$. 
Therefore, by equation \ref{eq:summation}, we have 
$$\frac{1}{p} \mu_p \circ c^{-1} (\{\omega \in B_+~:~\omega_0=\omega_1\})= \Big(\sum_{n=0}^\infty (1-p)^{2n}\Big)\mu_p \circ c^{-1}(B_+)=
\frac{1}{p(2-p)}\mu_p \circ c^{-1}(B_+),$$
as expected so that $B_+$ satisfies equation \ref{eq:conditional expectation2}. 
The case of $B_-$ is very much similar. But, observe that $c^{-1}(B_-) \not \in \sA$ up to measure zero, because the conditions that $\omega_0=-1$ and $\omega$ is zero-collapsible
imply that $\omega_1=-1$. However, we do have that $\sigma(B_-) \in \sA$. 
$$c^{-1}(B_+) \cap S_n^+=c^{-1}(B_+) \cap \Big(\bigcap_{i=0}^{2n-1} \set{\omega}{$\omega_i \neq \omega_{i+1}$}\Big) \cap \set{\omega}{$\omega_{2n} = \omega_{2n+1}$}.$$
The reader may check that $\mu_p \circ c^{-1}(B_- \cap S_n^-)=(1-p)^{2n} p \mu_p \circ c^{-1}(B_-),$ so that the same computation proves
equation \ref{eq:conditional expectation2} for $B$ replaced by $B_-$. Summing the equation from the $B_+$ case with the equation from the $B_-$ case proves equation \ref{eq:conditional expectation2} in general.
\end{proof}

Recall that if $c(\omega)=\eta$, then $\omega=\sI_f(\eta)$ for some $f:\Z \to \{(-+)^n~:~n \geq 0\}$. See Proposition \ref{prop:inverses of collapsing map}.
The function $\sI_f$ inserts the word $f(0)$ between $\eta_0$ and $\eta_1$. 
Note that $\iota:\omega \mapsto \frac{1}{2}\ell\circ f(0)$ defines a function from the set $C$ of collapsible elements of $\Omega_\pm$  to $\Z$. 
We would like to understand the conditional expectation of $\iota$ given $c(\omega)_0$ and $c(\omega)_1$.  
The reason for interest in this is given by Lemma \ref{lem:returns} and Corollary \ref{cor:matrix formula}.

\begin{lemma}
\name{lem:expectation for insertion length}
Suppose $0<p \leq 1$, and 
define the map $\iota:C \to \Z$ as above. Let $s_0, s_1 \in \{\pm 1\}$. Then the conditional expectation of $\iota$ given $c(\omega)_0=s_0$ and $c(\omega)_1=s_1$ is given by 
\begin{enumerate}
\item $\displaystyle 1+\frac{(1-p)^2}{p(2-p)}$ if $s_0 s_1=-+$, and
\item $\displaystyle \frac{(1-p)^2}{p(2-p)}$ otherwise.
\end{enumerate}
\begin{proof}
Let $n \in \Z$ with $n \geq 0$. By Corollary \ref{cor:preimages of cylinders}, the sets of the form 
$$U(s_0, s_1, n)=\set{\omega}{$c(\omega_0)=s_0$, $c(\omega_1)=s_1$, and $\iota(\omega)=n$}$$
are cylinder sets. (Observe the special case $U(-,+,0)=\emptyset$, by Corollary \ref{cor:preimages of cylinders}.)
Let $q=\frac{1}{2-p}$. Then since each preimage under $c$ of a point the cylinder set $\sC(s_0 s_1)$ lies in some $U(s_0,s_1,n)$ we
have 
$$\sum_{n=0}^\infty \mu_p\big(U(s_0,s_1,n)\big)=p\mu_q\big(\sC(s_0 s_1)\big)$$
The conditional expectation of $\iota=n$ is therefore given by 
$$C(s_0s_1)=\frac{1}{p\mu_q\big(\sC(s_0 s_1)\big)} \sum_{n=0}^\infty n\mu_p\big(U(s_0,s_1,n)\big).$$
We can apply Corollary \ref{cor:preimages of cylinders} to describe exactly which cylinder set $U(s_0,s_1,n)$ is, and
apply Proposition \ref{prop:measure uniquely defined} to compute its measure. This yields the following formulas.
$$C(-+)=\frac{2}{p(1-q)} \sum_{n=0}^\infty \frac{n}{2} p^2 (1-p)^{2n-1}. \qquad
C(+-)=\frac{2}{p(1-q)} \sum_{n=0}^\infty \frac{n}{2} p^2(1-p)^{2n+1}.$$
$$C(--)=C(++)=\frac{2}{pq} \sum_{n=0}^\infty \frac{n}{2} p^2(1-p)^{2n}.$$
\end{proof}
\end{lemma}

\subsection{The renormalization cocycle}
Let $0<p<1$ and $0<q<1$, and define $\nu=\mu_p \times \mu_q \times \mu_N$. This is a $\Phi$-invariant measure.
In this section, we pursue the philosophy laid out beneath the Renormalization Theorem \ref{thm:renormalization}
to exhibit a formula for the total measure of those $(\omega, \omega', \v)$ without stable periodic orbits.

Recall the definition of the embedding $\epsilon$ of $\R^6$ into the space of Borel measurable functions $\sM$ on $X$ given above equation \ref{eq:L}. 
We use $L$ to denote $\epsilon(\R^6)$. First, we would like to know how to compute $\int_X \epsilon(\bp)(x)~d\nu(x)$. This depends nicely on $p$ and $q$, and linearly on $\bp$.  

\begin{proposition}
\name{prop:measure vector}
Let $p$ and $q$ be as above and $\nu=\mu_p \times \mu_q \times \mu_N$. Then for any $\bp \in \R^6$ we have 
$$\int_X \epsilon(\bp)(x)~d\nu(x)=\m_{p,q} \cdot \bp \qquad \textrm{where $\m_{p,q}=(\frac{1-p}{4}, \frac{1-p}{4}, \frac{p}{2},\frac{1-q}{4}, \frac{1-q}{4}, \frac{q}{2}) \in \R^6$.}$$
\end{proposition}
\begin{proof}
The $i$-th entry of $\m_{p,q}$ is just $\nu(\Step_i)$. 
\end{proof}

We would also like to compute the matrix $M'$ described in Corollary \ref{cor:matrix formula}. 

\begin{proposition}
\name{prop:matrices random case}
We have $C(\nu,1) \circ \epsilon(\bp) \in L$ for all $\bp \in L$, and $\nu \circ \rho^{-1}$-a.e. we have $C(\nu,1) \circ \epsilon(\bp)=\epsilon \big( M_{p,q} \bp\big)$ where
$$\textstyle M_{p,q}=\left[\begin{array}{rrrrrr}
1 & 0 & 2 & 0 & 0 & 2 \\
0 & 1 & 0 & 0 & 0 & 0 \\
0 & 0 & 1 & 0 & 0 & 0 \\
0 & 0 & 2 & 1 & 0 & 2 \\
0 & 0 & 0 & 0 & 1 & 0 \\
0 & 0 & 0 & 0 & 0 & 1 \\
\end{array}\right]+
\frac{(1-p)^2}{p(2-p)} \left[\begin{array}{rrrrrr}
1 & 1 & 0 & 0 & 0 & 2 \\
1 & 1 & 0 & 0 & 0 & 2 \\
1 & 1 & 0 & 0 & 0 & 2 \\
0 & 0 & 0 & 0 & 0 & 0 \\
0 & 0 & 0 & 0 & 0 & 0 \\
0 & 0 & 0 & 0 & 0 & 0 \\
\end{array}\right]+
\frac{(1-q)^2}{q(2-q)} \left[\begin{array}{rrrrrr}
0 & 0 & 0 & 0 & 0 & 0 \\
0 & 0 & 0 & 0 & 0 & 0 \\
0 & 0 & 0 & 0 & 0 & 0 \\
0 & 0 & 2 & 1 & 1 & 0 \\
0 & 0 & 2 & 1 & 1 & 0 \\
0 & 0 & 2 & 1 & 1 & 0 \\
\end{array}\right].$$
\end{proposition}
\begin{proof}
Let $x=(\omega, \omega', \v) \in X$. 
To apply Corollary \ref{cor:matrix formula}, we need to know the conditional expectations for the return times $\ret{1}{x}$ given that $\rho(x) \in \Step_j$. By Lemma \ref{lem:returns},
this is equivalent to knowing the conditional expectation of one of $\ell \circ f(0)$, $\ell \circ f(-1)$, $\ell \circ f(0)$, or $\ell \circ f(-1)$. Here $f$ and $f'$ are
defined so that $\sI_f \circ c(\omega)=\omega$ and $\sI_{f'} \circ c(\omega')=\omega'$. By Lemma \ref{lem:expectation for insertion length} and shift invariance,
these numbers depend only the fact that $\rho(x)$ has step type $j$.  Moreover, the vector of values $\m=(m_1, \ldots, m_6)$ described in Corollary \ref{cor:matrix formula}
is given by 
$$\m=(1,0,0,1,0,0)+\frac{(1-p)^2}{p(2-p)}(1,1,1,0,0,0)+ \frac{(1-q)^2}{q(2-q)}(0,0,0,1,1,1).$$
Plugging these values into the matrix in Corollary \ref{cor:matrix formula} produces this proposition.
\end{proof}

Now we wish to restate Corollary \ref{cor:approach} in this context. Before we do this, we apply a change of coordinates.
For real numbers $m,n \geq 0$, define
\begin{equation}
\name{eq:coordinate change}
p(m)=\frac{m}{m+1} \and q(n)=\frac{n}{n+1}.
\end{equation}
Observe that we can restate the equation given in Theorem \ref{thm:img} as
$$\mu_{p(m)} \circ c^{-1}=\frac{m}{m+1} \mu_{p(m+1)}$$
So, by induction we have $\mu_{p(m)} \circ c^{-k}=\frac{m}{m+k} \mu_{p(m+k)}$.

Recall that $\NS$ denotes the collection of points in $X$ without stable periodic orbits.
\begin{lemma}
\name{lem:limit for random cases}
Suppose $m$ and $n$ are positive real numbers. Let $\nu=\mu_{p(m)} \times \mu_{q(n)} \times \mu_N$. 
Then
$$\nu(\NS)=\lim_{k \to \infty} \frac{mn}{(m+k)(n+k)} \m_{p(m+k),q(n+k)} \cdot \Big(\big(\prod_{i=k-1}^{0} M_{p(m+i),q(n+i)}\big) \1\Big),$$
where $\prod_{i=k-1}^{0} M_{p(m+i),q(n+i)}=M_{p(m+k-1),q(n+k-1)} \ldots M_{p(m+1),q(n+1)}M_{p(m),q(n)}$. In this formula, 
$\1 \in \R^6$ is the vector with all entries equal one, so that $\epsilon(\1)=\bbone$, the vector $\m_{p(m+k),q(n+k)}$ is defined as in Proposition
\ref{prop:measure vector}, and the matrices $M_{p(n+i),q(n+i)}$ are defined in Proposition \ref{prop:matrices random case}.
\end{lemma}

We already described all the tools necessary to reformulate Corollary \ref{cor:approach} in this context as above. 
Observe that the quantity in the limit corresponds to $\int_X C(\nu, n)(\bbone)~d\nu \circ \rho^{-n}$ translated under the embedding $\epsilon$. 
The only remark worth making is that 
$$C(\nu, n)(\bbone)=\epsilon \Big(\big( \prod_{i=k-1}^{0} M_{p(m+i),q(n+i)} \big)\1\Big)$$
by expanding out the cocycle as in Equation \ref{eq:expand cocycle}
and applying Proposition \ref{prop:matrices random case}. 

\subsection{Evaluating the limit}
\name{sect:evaluating the limit}
Our goal for the remainder of this section is to prove that the limit described in Lemma \ref{lem:limit for random cases} is zero. 
In this subsection, we introduce several simplifications which make this goal easier to attain.

We write $\sL(m,n)$ for the quantity $\nu(\NS)$ where $\nu=\mu_{p(m)} \times \mu_{q(n)} \times \mu_N$. Observe this is the limit described in the Lemma \ref{lem:limit for random cases}.

\begin{proposition}[Simplification 1]
If $m' \geq m$ and $n' \geq n$, then $\sL(m,n)=0$ implies  $\sL(m',n')=0$.
\end{proposition}
\begin{proof}
Observe that the functions $p(m)$ and $q(n)$ are increasing, and the functions $p \mapsto \frac{(1-p)^2}{p(2-p)}$ and $q \mapsto \frac{(1-q)^2}{q(2-q)}$ are decreasing. 
This second map is relevant to the definition of $M_{p(m),q(n)}$. Observe then that all entries of $M_{p(m+i),q(n+i)}$ are non-strictly larger
than those of  $M_{p(m'+i),q(n'+i)}$ for all $i$. Therefore, all entries of the product $\big( \prod_{i=k-1}^{0} M_{p(m+i),q(n+i)} \big)\1$
are larger than those of $\big( \prod_{i=k-1}^{0} M_{p(m'+i),q(n'+i)} \big)\1$. Also observe that the ratio
$$ \left. \frac{mn}{(m+k)(n+k)} \m_{p(m+k),q(n+k)} \middle/ \frac{m'n'}{(m'+k)(n'+k)} \m_{p(m'+k),q(n'+k)} \right. $$
tends to a non-zero constant vector as $k \to \infty$,  with vector division interpreted coodinatewise. 
\end{proof}

\begin{proposition}[Simplification 2]
$\sL(m+1,n+1)=0$ implies  $\sL(m,n)=0$.
\end{proposition}
We omit the proof, because it is simpler than the last. The reader should observe that 
pull out the right-most matrix in the product for the $\sL(m,n)$ case obtains a product relevant to the $\sL(m+1,n+1)$ case.

Combining these two propositions yields the following:
\begin{lemma}[Combined Simplification]
\name{lem:combined simplification}
If $\sL(1,1)=0$, then $\sL(m,n)=0$ for all $m,n>0$. 
\end{lemma}

It turns out that there are some symmetries which appear in case we want to evaluate a limit of the form $\sL(n,n)$. 
Consider the linear projection $\pi:\R^6 \to \R^2$ and section $\s:\R^2 \to \R^6$ satisfying $\bpi \circ s=\id$  given by 
\begin{equation}
\name{eq:pi}
\bpi(a,b,c,d,e,f)=(a+b+d+e,c+f) \and \s(a,b)=(\frac{a}{4}, \frac{a}{4}, \frac{b}{2}, \frac{a}{4}, \frac{a}{4}, \frac{b}{2})
\end{equation}
We think of $\s$ and $\pi$ as sending row vectors to row vectors.
We observe that for all $\bw \in \R^2$ and all $n$ we have 
\begin{equation}
\name{eq:N}
\s(\bw) M_{p(n),q(n)}=\s(\bw N_n) \quad \textrm{where} \quad N_n=\left[\begin{array}{rr}
1 & 2 \\
0 & 1
\end{array}\right]+ \frac{2}{n(n+2)}
\left[\begin{array}{rr}
1 & 1 \\
1 & 1
\end{array}\right].
\end{equation}
Therefore, we can write
\begin{equation}
\name{eq:simp3}
\sL(1,1)=\lim_{k \to \infty} \frac{1}{(1+k)^2} \left[\begin{array}{rr} \frac{1}{k+1} & \frac{k}{k+1} \end{array}\right] \big(\prod_{n=k}^1 N_n \big)\left[\begin{array}{r} 1 \\ 1 \end{array}\right],
\end{equation}
where $\prod_{n=k}^1 N_n=N_k N_{k-1} \ldots N_1$. 

We will further simplify this expression. Define the matrices
\begin{equation}
\name{eq:AB}
A=\left[\begin{array}{rr}
1 & 2 \\
0 & 1
\end{array}\right],
\quad 
B=\left[\begin{array}{rr}
1 & 1 \\
1 & 1
\end{array}\right],
\and
B_n=\frac{1}{n}B.
\end{equation}

\begin{lemma}[Final Simplification]
\name{lem:final simplification}
$$\sL(1,1)=\frac{1}{12} \lim_{k \to \infty} \left[\begin{array}{rr} \frac{1}{k+1} & \frac{k}{k+1} \end{array}\right] 
\prod_{n=k}^1 (\frac{n+2}{n+4} A + \frac{2}{n+4} B_n) \left[\begin{array}{r} 1 \\ 1 \end{array}\right],$$
where the product expands as in Lemma \ref{lem:limit for random cases} (for instance). 
\end{lemma}
\begin{proof}
We will move the $\frac{1}{(1+k)^2}$ factor into the product in Equation \ref{eq:simp3}. Observe that
$\lim_{k \to \infty} \frac{(k+3)(k+4)}{(1+k)^2}=1.$
Furthermore, we can write
$\frac{12}{(k+3)(k+4)}=\prod_{n=1}^k \frac{n+2}{n+4}.$ Therefore we can rewrite Equation \ref{eq:simp3} as desired:
$$\sL(1,1)=\frac{1}{12} \lim_{k \to \infty} \left[\begin{array}{rr} \frac{1}{k+1} & \frac{k}{k+1} \end{array}\right] 
\prod_{n=k}^1 (\frac{n+2}{n+4} N_n) \left[\begin{array}{r} 1 \\ 1 \end{array}\right].$$
\end{proof}

Finally, we can prove the following theorem, which is just a restatement of Theorem \ref{thm:random tilings}.

\begin{theorem}
\name{thm:random tilings2}
Let $0<p,q<1$ and let $\nu=\mu_{p} \times \mu_{q} \times \mu_N$. Then
$\nu(\NS)=0$. That is, $\nu$-almost every $x \in X$ has a stable periodic orbit under $\Phi$. 
\end{theorem}
\begin{proof}
By Lemmas \ref{lem:limit for random cases} and \ref{lem:combined simplification}, it is sufficient to prove
that $\sL(1,1)=0$. This amounts to evaluating the limit in Lemma \ref{lem:final simplification}. The main idea is to think of expanding out the product $\prod_{n=k}^1 (\frac{n+2}{n+4} A + \frac{2}{n+4} B_n)$ into $2^k$ terms weighted by constants. Observe that the vector $(\frac{n+2}{n+4}, \frac{2}{n+4})$ is a probability vector, so that the sum of these constants will always be $1$ for all $k$.

To simplify notation we write $\v_k$ to denote the row vector $(\frac{1}{k+1}, \frac{k}{k+1})$ and 
write $\1$ to denote the column vector $(1,1)$. 

Consider the alphabet $\sA= \{a,b\}$. Given an integer $k \geq 0$, we think of an element $\alpha \in \sA^k$ as a map $\{1, \ldots, k\} \to \sA$ given by $n \mapsto \alpha_k$. 
For each $k \geq 0$ we have a the natural restriction map $r_k:\bigcup_{j=0}^\infty \sA^{k+j} \to \sA^k$.
We can write elements of $\alpha$ as word in $\sA$ with decreasing index. For instance
$\alpha=baa$ indicates that $\alpha_3=b$ and $\alpha_2=\alpha_1=a$. 
In particular, if $\alpha \in \sA^k$, we use $a\alpha$ and $b \alpha$ to denote the elements of $\sA^{k+1}$ so that
$r_k(a\alpha)=r_k(b\alpha)=\alpha$, $(a\alpha )_{k+1}=a$ and $(b\alpha)_{k+1}=b$.

We define a function $M$ from $\N \times \sA$ to the space of $2 \times 2$ matrices by 
$$M(n,a)=A \and M(n,b)=B_n,$$
where $A$ and $B_n$ are defined as in Equation \ref{eq:AB}. Given $\alpha \in \sA^k$ we define
$$M(\alpha)=\prod_{n=k}^1 M(n,\alpha_n)=M(k, \alpha_k) M(k-1, \alpha_{k-1}) \ldots M(1, \alpha_1).$$  
We define $M(\emptyset)$ to be the identity matrix, where $\emptyset \in \sA^0$ represents the empty word.

The spaces $\sA^k$ come with probability measures $\nu_k$. To define this measure consider the function 
$w: \N \times \sA \to \R$ given by $w(n,a)=\frac{n+2}{n+4}$ and $w(n,b)=\frac{2}{n+4}$. Since $\sA^k$ is finite, we can describe the measure $\nu_k$ by the rule $\nu_k( \{\alpha\})=\prod_{n=1}^k w(n,\alpha_n)$. Observe that these measures behave nicely with the restriction map as 
\begin{equation}
\nu_k=\nu_{k+j} \circ (r_k\big|_{\sA^{k+j}})^{-1}.
\end{equation}
Our reason for studying these objects is that the product we would like to take a limit 
of can be expressed as an integral for all $k$:
\begin{equation}
\name{eq:prod to integral}
\v_k \big(\prod_{n=k}^1 (\frac{n+2}{n+4} A + \frac{2}{n+4} B_n)\big) \1=\int_{\sA^k} \v_k M(\alpha) \1~d \nu_k(\alpha).
\end{equation}
To simplify notation, we define
$f_k(\alpha)=\v_k M(\alpha) \1 \in \R$ whenever $\alpha \in \sA^k$. 

We make some basic observations about these functions $f_k$. First,
\begin{equation}
\name{eq:s2}
\textrm{If $\alpha \in \sA^k$ and $\alpha_k=b$, then $M(\alpha) \1=\big(f_k(\alpha), f_k(\alpha)\big)$.}
\end{equation}
To see this observe that $M(\alpha)=B_k M \circ r_{k-1}(\alpha)$. By the structure of $B_k$ we have
that for any $\bw \in \R^2$, $B_k\bw=(c,c)$ for some $c \in \R$. So we can write  $M(\alpha) \1=(c,c)$ for some $c$. 
Then since $\v_k$ is a probability vector, $f_k(\alpha)=\v_k M(\alpha) \1=c$. Our second observation is:
\begin{equation}
\name{eq:s3}
\textrm{If $\alpha \in \sA^k$ and $\alpha_k=b$, then $j \mapsto f_{k+j}(a^j \alpha )$ for $j \geq 0$ is bounded and increasing.}
\end{equation}
To see this, observe that Statement \ref{eq:s2} says that $M(\alpha) \1=(c,c)$ for $c=f_k(\alpha)>0$. Then,
$M(a^j \alpha)=\big((2j+1)c,c\big)$ and so 
$$f_{k+j}(a^j \alpha)=(\frac{1}{k+j+1},\frac{k+j}{k+j+1}) \cdot \big((2j+1)c,c\big)=\frac{(3j+k+1)c}{k+j+1},$$
which is an increasing sequence converging to $3c$ as $j \to \infty$. 
In fact, since $f(\alpha)=c$, we have proved that
\begin{equation}
\name{eq:s4}
\textrm{If $\alpha \in \sA^k$ and $\alpha_k=b$, then $f_{k+j}(a^j \alpha) \to 3 f_k(\alpha)$ as $j \to \infty$.}
\end{equation}

Statements  \ref{eq:s3} and  \ref{eq:s4} restrict how the value of $f_{k+j}(\alpha a^j)$ grows as $j \to \infty$. In particular,
to exceed $3 f_k(\alpha)$ by appending letters to $\alpha$, $b$'s must be appended at some point. We will study the times at which $b$'s are appended.
Given $\alpha \in \sA^k$ define $L(\alpha)$ to be the largest $j$ with $1 \leq j \leq k$ for which $\alpha_j=b$, and define $L(\alpha)=0$ if no such $j$ exists.
Statements  \ref{eq:s3} and  \ref{eq:s4} then imply that for any $\alpha \in \sA^k$ we have
$$f_k(\alpha)<3 f_{L(\alpha)} \circ r_{L(\alpha)}(\alpha).$$
Note the right side is an evaluation of $f$ on a word terminating in a $b$. 
To understand the behavior of $f$ evaluated on such words we define the sequence of constants $\beta_k$ to be the average value of $f_k(\alpha)$ taken over all $\alpha \in \sA^k$ with $\alpha_k=b$. That is, $\beta_0=1$ and for $k>0$ we have
\begin{equation}
\name{eq:beta}
\beta_k=\frac{1}{\nu_k\{\alpha \in \sA^k|\alpha_k=b\}} \int_{\{\alpha \in \sA^k|\alpha_k=b\}} f_k(\alpha)~d \nu_k(\alpha).
\end{equation}
We can write the integral we would like to understand can be written in terms of the $\beta$'s. Observe,
$$\int_{\sA^k} f_k~d \nu_k<3 \sum_{n=0}^k P\big(L(\alpha)=n\big) \int_{\{\alpha' \in \sA^n~:~\alpha'_n=b\}}  f_{n}(\alpha')~d\nu_n(\alpha'),$$
where $P\big(L(\alpha)=n\big)=\nu_{k} \{\alpha \in \sA^{k}|\big(L(\alpha)=n\big\}$ denotes the $\nu_k$ probability that $L(\alpha)=n$. For $n>0$, 
we can evaluate this as
\begin{equation}
\name{eq:propL}
P\big(L(\alpha)=n\big)=\Big(\prod_{j=n+1}^{k} \frac{j+2}{j+4}\Big)\frac{2}{n+4}=\frac{2(n+3)}{(k+3)(k+4)}.
\end{equation}
In the case of $n=0$, we have $P\big(L(\alpha)=n)\big)=\frac{12}{(k+3)(k+4)}.$ 
Also, the integral in the sum is just $\beta_n$ so that 
$$\int_{\sA^k} f_k~d \nu_k<3 \left(\frac{12}{(k+3)(k+4)} \beta_0+\sum_{n=1}^k \frac{2(n+3)}{(k+3)(k+4)} \beta_n\right).$$
Observe the coefficients sum to one since $\sum_{n=0}^{k}  P\big(L(\alpha)=n)\big)=1$. Also, the coefficients of $\beta_n$ tend
to zero as $k \to \infty$. Therefore, we have
\begin{equation}
\name{eq:reduction to beta}\limsup_{k \to \infty} \int_{\sA^k} f_k~d \nu_k \leq \limsup_{k \to \infty} \beta_k.
\end{equation}
So, proving $\lim_{k \to \infty} \beta_k=0$ implies the theorem by Equation \ref{eq:prod to integral} and Lemma \ref{lem:final simplification}.

Now we wish to find an inductive formula for $\beta_k$. Observe that for all $\alpha' \in \sA^{k-1}$,  
$$\frac{1}{\nu_k\{\alpha \in \sA^k|\alpha_k=b\}} \nu_k(b\alpha')=\nu_{k-1}(\alpha').$$
Therefore, for $k>0$, we can simplify equation \ref{eq:beta} as
\begin{equation}
\name{eq:beta2}
\beta_k=\int_{\sA^{k-1}} f_k(b\alpha')~d \nu_{k-1}(\alpha').
\end{equation}
Suppose $\alpha=b\alpha'$ with $\alpha' \in \sA^{k-1}$. Let $n=L(\alpha')$,  
$\alpha''=r_{L(\alpha')}(\alpha)$, and $c=f_{n}(\alpha'')$. Note that $\alpha''$ is $\alpha$ with the last $b$ and the longest prior string of $a$'s
removed. We have
$$M(\alpha) \1=B_k A^{k-n-1} (c,c)=B_k \big((2k-2n-1)c,c\big)=\frac{2k-2n}{k} (c,c).$$
And so by Statement \ref{eq:s2}, we know
\begin{equation}
f_k(\alpha)=\frac{2k-2n}{k} f_n(\alpha'').
\end{equation}
Also for $n>0$, the probability that $L(\alpha')=n$ is given by $P\big(L(\alpha'=n)\big)=\frac{2(n+3)}{(k+2)(k+3)}$ by Equation \ref{eq:propL}.
And, in the special case of $n=0$ we have $P\big(L(\alpha'=0)\big)=\frac{12}{(k+2)(k+3)}$.
So we can evaluate the integral in Equation \ref{eq:beta2} as
$$\int_{\sA^{k-1}} f_k(b\alpha')~d \nu_{k-1}(\alpha')=\sum_{n=0}^{k-1} \frac{2k-2n}{k} P\big(L(\alpha'=n)\big) \int_{\{\alpha'' \in \sA^n|\alpha_n=b\}} f_n(\alpha'')~d \nu_n(\alpha'').$$
Plugging this and the values of $P\big(L(\alpha'=n)\big)$ into Equation \ref{eq:beta}, we obtain that for $k>0$,
\begin{equation}
\beta_k=\displaystyle \frac{24}{(k+2)(k+3)} \beta_0+\sum_{n=1}^{k-1} \frac{2(n+3)(2k-2n)}{k(k+2)(k+3)} \beta_n
\end{equation}
To avoid special treatment of $\beta_0$, we define $\gamma_k=\beta_k$ for $k>0$ and set $\gamma_0=2 \beta_0=2$. Then for $k>0$ we have
\begin{equation}
\name{eq:gamma sum}
\gamma_k=\sum_{n=0}^{k-1} \frac{2(n+3)(2k-2n)}{k(k+2)(k+3)} \gamma_n
\end{equation}
Define the constant $s_k$ for $k>0$ to be the sum of the coefficients of the above expression. That is,
\begin{equation}
\name{eq:s k}
s_k=\sum_{n=0}^{k-1} \frac{2(n+3)(2k-2n)}{k(k+2)(k+3)}=
\frac{2(k+1)(k+8)}{3(k+2)(k+3)}.
\end{equation}
Observe that $s_k \leq 1$ for all $k$. In particular, $\gamma_k \leq 2 s_k$ for all $k$. 
To see that $\gamma_k \to 0$ as $k \to \infty$, we will inductively prove that for all $m \geq 0$
\begin{equation}
\name{eq:inductive statement}
\textrm{For any $m>0$, there is a $K_m$ such that $\gamma_k \leq 2 (\frac{8}{9})^m$ for $k \geq K_m$.}
\end{equation}
This will prove the theorem by the remarks surrounding Equation \ref{eq:reduction to beta}.
The above argument proves this is true for $m=0$ with $K_m=0$. Now suppose Statement \ref{eq:inductive statement} is true for $m-1$. We will prove it for $m$. 
Observe that if you fix any $n$, the coefficients of $\gamma_n$ in the sum of Equation 
\ref{eq:gamma sum} tend to zero as $k \to \infty$. Then, there is a constant $J_m \geq K_{m-1}$ so that $k \geq J_m$ implies
$$\sum_{n=0}^{K_{m-1}-1} \frac{2(n+3)(2k-2n)}{k(k+2)(k+3)} \gamma_n \leq \frac{2}{9}(\frac{8}{9})^{m-1}. $$
By the inductive hypothesis, the remainder of the sum of Equation \ref{eq:gamma sum} satisfies 
$$\sum_{n=K_{m-1}}^{k-1} \frac{2(n+3)(2k-2n)}{k(k+2)(k+3)} \gamma_n \leq 2 (\frac{8}{9})^{m-1} \sum_{n=K_{m-1}}^{k-1} \frac{2(n+3)(2k-2n)}{k(k+2)(k+3)}.$$
Observe that this second sum is bounded from above by $s_k$ which tends to $\frac{2}{3}$ as $k \to \infty$. See Equation \ref{eq:s k}. We conclude that there is a constant $K_m \geq J_m$
so that $k \geq K_m$ implies 
$$\sum_{n=K_{m-1}}^{k-1} \frac{2(n+3)(2k-2n)}{k(k+2)(k+3)} \gamma_n \leq \frac{14}{9} (\frac{8}{9})^{m-1}.$$
Putting these two statements together yields that for $k \geq K_m$ we have 
$$\gamma_k=\sum_{n=0}^{k-1} \frac{2(n+3)(2k-2n)}{k(k+2)(k+3)} \gamma_n \leq \frac{16}{9} (\frac{8}{9})^{m-1}=2 (\frac{8}{9})^m.$$
\end{proof}

\bibliographystyle{amsalpha}
\bibliography{/home/pat/active/my_papers/bibliography}
\end{document}